\documentclass[final]{siamart171218}

\usepackage{macros,graphicx}
\usepackage{epstopdf}
\usepackage{url}
\usepackage{color}

\usepackage{subcaption}

\def\norm#1{\|#1\|}

\newcommand{\EE}[1]{\mathrm{E}\left \{ #1\right\}}	

\usepackage{mathtools}

\newcommand{\sumn}{{\sum_{k=0}^{N_\epsilon-1} }}

\usepackage{subcaption}
\usepackage{algorithm}
\usepackage[noend]{algpseudocode}

\newtheorem{assumption}[theorem]{Assumption}
\newtheorem{remark}[theorem]{Remark}

\newcommand{\ab}[1]{{\color{black}#1}}
\newcommand{\change}[1]{{\color{black}#1}}
\definecolor{LiyuanGreen}{rgb}{0,0.7,0}

\usepackage[numbers, sort&compress ]{natbib}
\numberwithin{equation}{section}
\usepackage{bbm}
\usepackage[normalem]{ulem}

\newcommand{\Alpha}{{\cal A}}
\def\EE{\mathbb{E}}

\usepackage{booktabs}
\usepackage{array}
\usepackage{arydshln}
\title {Global Convergence Rate Analysis of a Generic Line Search Algorithm with Noise\footnotemark[1]}
\author{A.~S.~Berahas\footnotemark[2] \and L.~Cao\footnotemark[3] \and K.~Scheinberg\footnotemark[4]}
\date  {\today}

\begin{document}

\maketitle

\renewcommand{\thefootnote}{\fnsymbol{footnote}}
\footnotetext[1]{This work was  partially supported by NSF Grants CCF 16-18717 and TRIPODS 17-40796, DARPA Lagrange award HR-001117S0039 and a Google Faculty Award.}
\footnotetext[2]{Department of Industrial and Operations Engineering, University of Michigan, Ann Arbor, MI, USA; E-mail: \email{albertberahas@gmail.com}}
\footnotetext[3]{Department of Industrial and Systems Engineering, Lehigh University, Bethlehem, PA, USA; E-mail: \email{liyuan@lehigh.edu}}
\footnotetext[4]{School of Operations Research and Information Engineering, Cornell University, Ithaca, NY, USA; E-mail: \email{katyas@cornell.edu}}
\renewcommand{\thefootnote}{\arabic{footnote}}

%
%
%
%

\begin{abstract}
  In this paper, we develop convergence analysis of a modified line search method for objective functions whose value is computed with noise and whose gradient estimates are inexact and possibly random. The noise is assumed to be bounded in absolute value without any additional assumptions.  We extend the framework based on stochastic methods from \cite{cartis2018global} which was developed to provide analysis of a standard line search method with exact function values and random gradients to the case of noisy functions. We introduce \ab{two alternative} conditions on the gradient  which  when  satisfied with some sufficiently large probability at each iteration, guarantees convergence properties of the line search method. We derive expected complexity bounds \ab{ to reach a near optimal neighborhood } for convex, strongly convex and nonconvex functions. \ab{The exact   dependence of the convergence neighborhood  on the noise  is specified. }
\end{abstract}

\section{Introduction}

We consider an unconstrained optimization problem of the form
\begin{align}		\label{eq.prob}
	\min_{x \in \mathbb{R}^n} \phi(x), 
\end{align}
where $f(x\ab{,\xi})=\phi(x) + \ab{e}(x\ab{,\xi})$ is computable, while $\phi(x)$ is not\ab{, and $\xi$ is a random variable with associated probability space $(\Xi, \mathcal{F},P)$}. In other words $f: \mathbb{R}^n \ab{\times \Xi}\rightarrow \mathbb{R}$ is a possibly noisy approximation of a smooth function $\phi: \mathbb{R}^n\rightarrow \mathbb{R}$, and the goal is to minimize $\phi$.  Alternatively, $f(x)$ may be a nonsmooth function and $\phi(x)$ its smooth approximation; see for instance \cite{nesterov2017random, maggiar2018derivative}.  Such problems arise in a plethora of fields such as Derivative-Free Optimization (DFO) \cite{ConnScheVice08c,LarsMeniWild2019}, Simulation Optimization \cite{pasupathy2018sampling} and Machine Learning. There has been a lot of work analyzing the case when \ab{$e:\mathbb{R}^n\times \Xi \rightarrow \mathbb{R}$} is a random function with zero mean. Here, we take a different research direction, allowing $\ab{e}(x\ab{,\xi})$ to be stochastic, deterministic, or adversarial, but assuming that $|\ab{e}(x\ab{,\xi})|\leq \epsilon_f$ for all $x \in  \mathbb{R}^n$ \ab{and all realizations of $\xi$}. While this is a strong assumption, it is often satisfied in practice when $f(x\ab{,\xi})$ is a result of a computer code aimed at computing \ab{(or approximating)} $\phi(x)$, but has inaccuracies due to internal discretization \cite{MoreWild09,more2011estimating}. It will be evident from our analysis, that the modified line search method makes progress as long as $\|\nabla \phi(x)\|$ is sufficiently large compared to the noise. 

Line searches are classical and well-known techniques for improving the performance of optimization algorithms \cite{NoceWrig06}. They allow algorithms to be more robust, less dependent on the choices of hyper-parameters and typically ensure faster practical convergence rates. However, in their original form they rely on exact function and gradient information. Many modern applications give rise to functions for which computing accurate function values and/or gradients is either impossible or prohibitively expensive.  Thus, it is desirable to extend the line search paradigm and its analysis to such functions.  In \cite{cartis2018global} a general line search algorithm was analyzed under the conditions that the function values are {\em exact} but the gradient estimates are {\em  inexact} and {\em random}. 
It is shown that under certain (realizable) probabilistic conditions on the accuracy of the gradient estimates, the resulting line search has the same expected complexity  (up to constants)  as the line search based on exact gradients.  In \cite{BlanchetCartisMenickellyScheinberg2019} a general framework for complexity analysis  of stochastic optimization methods is proposed and applied to a  trust region method.  The same framework is used in \cite{paquette2018stochastic}  to analyze a  line search method applied to {\em stochastic} functions.   This framework  is significantly more complicated than that in \cite{cartis2018global}, but it also relies on casting the algorithm as a stochastic process (a submartingale) and it is again shown that  the expected complexity is the same (up to constants) as that of  regular deterministic gradient descent, under certain probabilistic, and realizable,  conditions on stochastic function values and gradient estimates. 
 
In this paper, we extend the analysis in \cite{cartis2018global} to apply to \eqref{eq.prob}. In particular, we assume that the gradient estimates are random and the function values are noisy. Since the function values are noisy (unlike in \cite{cartis2018global}), the line search is modified to accept steps that may potentially increase the current estimated value. This modification causes significant changes in the analysis of the expected complexity rates, as the analysis in \cite{cartis2018global} heavily relies on the fact that objective function can never be increased by the algorithm. Nevertheless, we are able to extend the results in  \cite{cartis2018global} recovering expected complexity bounds for the cases of convex, strongly convex and nonconvex objective functions. \ab{We note here that we derived the complexity bounds for the condition on the gradient accuracy used in \cite{cartis2018global}, as well as for the so-called {\em norm condition }
used, for example, in \cite{byrd2012sample}.  While, as we discuss later, each gradient accuracy condition can have advantages over the other, depending 
on the setting,  they  can be used interchangeably with relatively small adjustment to the analysis. Specifically, the analysis of the supermartingale is not affected by the choice of this condition, and the key steps of the analysis of the line search method itself are analogous, however, the constants stemming from the gradient condition appear differently in the final complexity bounds.}

The conditions we impose on the line search algorithm are essentially the same as in  \cite{cartis2018global}, while the conditions required for the analysis of the stochastic line search \cite{paquette2018stochastic}, where the noise is unbounded, are more restrictive and thus that analysis does not apply  to the case we consider here.  In particular, while the function value noise is allowed to be unbounded in   \cite{paquette2018stochastic}, it is assumed that it is possible to reduce its variance below any given threshold, for example, by sample averaging. In contrast here we do not assume that the noise is stochastic, thus we do not assume it can be reduced or controlled. Also, the algorithm itself in \cite{paquette2018stochastic}  is more complicated than a simple line search in order to handle unbounded noise.  Moreover, the resulting bounds in \cite{paquette2018stochastic} have worse dependence on constants than those in \cite{cartis2018global} and the bounds we derive in this paper. Finally, in both \cite{BlanchetCartisMenickellyScheinberg2019} and \cite{paquette2018stochastic}  the expected complexity bound is derived for any $\epsilon$, arbitrarily small, under the assumption that the noise can be made arbitrarily small accordingly (at least with sufficient probability). Here we establish  a connection between the level of noise and the convergence neighborhood. Obtaining similar results for  the setting in \cite{BlanchetCartisMenickellyScheinberg2019,paquette2018stochastic}  is nontrivial and  is a subject of future work. 

Our main motivation for this analysis is the recent popularity of smoothing methods for gradient estimates of black-box functions.  Stochastic gradient approximations can be computed at relatively low costs, e.g., via  Gaussian smoothing \cite{flaxman2005online,ES,nesterov2017random} and smoothing on a unit sphere \cite{fazel2018global}, and used within a gradient descent algorithm. This approach has been  analyzed in \cite{nesterov2017random} and more recently used  in several papers for the specific cases  of policy optimization in reinforcement learning and online learning \cite{flaxman2005online, ES, fazel2018global}.  All of these papers employ specific fixed step length gradient descent schemes within limited settings (e.g., convex functions). Our goal is to  develop convergence rate analyses for convex, strongly convex and nonconvex functions, for a generic line search algorithm based on gradient approximations,  that can apply not only to gradient descent, but also to quasi-Newton methods such as L-BFGS \cite{NoceWrig06}.

It turns out that the variance of the stochastic gradients computed via Gaussian and unit sphere smoothing can be bounded from above by the squared norm of the expectation, that is $\|\nabla \phi(x)\|^2$, when  $\phi$ is the smoothing function \cite{berahas2019theoretical}.  \ab{This motivates us  to consider a simpler probabilistic condition on the accuracy of the gradient estimates  in addition to the one used in  \cite{cartis2018global}.  }

\paragraph{Assumptions} Throughout the paper we make the following assumptions. 

\begin{assumption}	\label{assum:lip_cont} 
\textbf{(Lipschitz continuity of the gradients of $\pmb{\phi}$)} 
The function $\phi$ is continuously differentiable, and the gradient of $\phi$ is $L$-Lipschitz continuous for all $x \in \mathbb{R}^n$. 
\end{assumption}

\begin{assumption}	\label{assum:low_bound} \textbf{(Lower bound on $\pmb{\phi}$)} The function $\phi$ is bounded below by a scalar $\hat{\phi}$. 	
\end{assumption}

\begin{assumption} \label{assum:func_bnd} \textbf{(Boundedness of Noise in the Function)} There is a constant $\epsilon_f \geq 0$ such that $ | f(x\ab{,\xi}) - \phi(x)| = |e(x\ab{,\xi})| \leq \epsilon_f$ for all $x \in \mathbb{R}^n$ \ab{and all realizations of $\xi$}.
\end{assumption}

Assumption \ref{assum:func_bnd} may seem very strong, however, we will show that under this assumption the modified line search algorithm converges to a neighborhood of the optimal solution whose size is defined by $\epsilon_f$. Thus, if it is possible to control the value of $\epsilon_f$, then one can tighten the convergence neighborhood. This is possible in many applications, where for example, values of $\phi(x)$ are obtained as a limit to some discretized computation and the error is controlled by the fineness of the discrete grid \cite{MoreWild09,more2011estimating} or if $\phi(x)$ is a smoothed approximation of $f(x\ab{,\xi})$ where the smoothing parameter controls the error between $f(x\ab{,\xi})$ and $\phi(x)$ \cite{maggiar2018derivative,nesterov2017random}. \ab{We stress here that our algorithms and analysis  do {\em not} assume that the noise is stochastic or that the bound $\epsilon_f$ is controllable, just that it is known}.  

\paragraph{Summary of Results}
While we are motivated by some specific methods of computing gradient estimates, in the remainder of the paper, we simply aim to establish complexity bounds on a generic modified line search algorithm applied to the minimization of convex, strongly convex and nonconvex functions, under the condition that the gradient estimate \ab{$g: \mathbb{R}^n \rightarrow \mathbb{R}^n$} satisfies 
\begin{align}	\label{eq:norm}
	\| g(x) - \nabla \phi(x) \| \leq  \theta\|\nabla \phi(x)\|,
\end{align}
for sufficiently small $\theta$ with some probability $1-\delta$.\footnote{The norms used in this paper are Euclidean norms.}  The bound \eqref{eq:norm},  known as a the \emph{norm condition}, was first introduced in \cite{carter1991global} and consequently used in a variety of works (see e.g., \cite{byrd2012sample}). This bound is generally not realizable for generic stochastic gradients estimates, however, it can be made to hold for several deterministic and stochastic gradient estimates such as those used in \cite{berahas2019derivative,nesterov2017random,ConnScheVice08c,fazel2018global}. We establish expected complexity bounds similar to those in \cite{cartis2018global}, where the line search is analyzed under a more complicated bound on $\| g(x) - \nabla \phi(x) \|$ using exact evaluations of $\phi$ (i.e., no noise in the function evaluations). The expected complexity bounds are established in terms of desired accuracy $\epsilon$, under the assumption that $\epsilon$ is sufficiently big compared to the error level $\epsilon_f$. We derive specific bounds on $\epsilon$ with respect to $\epsilon_f$ for convex, strongly convex and nonconvex cases \ab{first for a gradient descent-type algorithm, and then for an algorithm that uses any general descent direction. For completeness, we derive complexity bounds for the condition on the gradient accuracy presented in \cite{cartis2018global}, however, due to the presense of noise, this condition is somewhat modified.}

\paragraph{Organization} The paper is organized as follows. In Section \ref{ls_alg} we describe a general line search algorithm that uses gradient approximations in lieu of the true gradient, and noisy function evaluations of the objective function. We present the stochastic analysis that allows us to bound the expected number of steps required by our generic scheme to reach a desired accuracy  in Section \ref{stoch_analysis}. This analysis is an extension of the results in \cite{cartis2018global} that accounts for noise in the objective function. In Section \ref{convergence_analysis}, we apply the results of  Section \ref{stoch_analysis} to derive global  convergence rates and bounds on $\epsilon$ in terms of $\epsilon_f$ when the generic line search method is applied to convex, strongly convex and nonconvex functions. Finally, in Section \ref{final_remarks} we make some concluding remarks and discuss avenues for future research.

\section{A Generic Modified Line Search Algorithm} \label{ls_alg}

In this section, we describe a generic line search algorithm that uses gradient approximations in lieu of the true gradient and that operates in the noisy regime. In general, line search algorithms construct a possibly noisy approximation of the gradient \ab{at the current iterate $x_k$, $g_k = g(x_k)$,} and compute a search direction using this gradient estimate and possibly additional information, e.g., a quasi-Newton search direction. The step size parameter is then chosen; this could be constant, selected from a predetermined sequence of step lengths (e.g., diminishing) or adaptive (e.g., via a back-tracking Armijo line search \cite[Chapter 3]{NoceWrig06}). The framework of the generic line search method we analyze is given in Algorithm \ref{alg:grad_approx_ls}. As is clear from Algorithm \ref{alg:grad_approx_ls}, the key components of this method are: $(i)$ the construction of the gradient approximation (Step \ref{step:grad_approx}); $(ii)$ the choice of the search direction (Step \ref{step:seach_dir}); and $(iii)$ the choice of the step size parameter and the iterate update (Step \ref{step:step_size}). 

\begin{algorithm}[ht]
  \caption{~\textbf{Generic Line Search Algorithm}}
  \label{alg:grad_approx_ls}
  {\bf Inputs:} Starting point $x_0$, initial step size parameter $\alpha_0 >0$.
  \begin{algorithmic}[1]
    \For{$k=0,1,2,\dots$}
        \State  {\bf Construct a gradient approximation $\ab{g_k}$:} \label{step:grad_approx}
        
        \hspace{0.5cm} Construct an approximation $\ab{g_k}$ of $\nabla \phi(x_k)$.
        \State {\bf Construct a search direction $d_k$:} \label{step:seach_dir}
        
        \hspace{0.5cm} Construct a search direction $d_k$, e.g., $d_k = - \ab{g_k}$ or $d_k = - H_k \ab{g_k}$.
        \State {\bf Compute step size $\alpha_k$ and update the iterate:} \label{step:step_size}
    \EndFor
  \end{algorithmic}
\end{algorithm}


Algorithm \ref{alg:grad_approx_ls} is a generic line search algorithm. We perform the analysis in Section \ref{convergence_analysis} for the case where $d_k = - \ab{g_k}$, and then outline how the analysis can be easily modified to the case of a more general search direction $d_k$, under additional assumptions on $d_k$.  In order to prove theoretical convergence guarantees, we need to fully specify the manner in which the step size parameter is selected at every iteration and how a new iterate is computed (Line \ref{step:step_size}).  We consider Algorithm \ref{alg:grad_approx_ls} for which the step size parameter $\alpha_k$ varies under the condition that $\alpha_k$  is  chosen to satisfy a \emph{modified} version of the  sufficient decrease Armijo condition,
\begin{align}		\label{eq:Armijo}
	f(x_k + \alpha_k d_k,\ab{\xi}) \leq f(x_k,\ab{\xi}) + c_1 \alpha_k d_k^T \ab{g_k} + 2\epsilon_f,
\end{align}
where $c_1 \in (0,1)$ is the Armijo parameter, and $\epsilon_f$ is \ab{the upper bound on the noise}  in the objective function. \ab{Note that the random
variable $\xi$ may have two different realizations when  computing $f(x_k + \alpha_k d_k,\ab{\xi}) $ and $f(x_k,\ab{\xi})$, however, these realizations may be dependent, independent or identical. This does not affect our analysis, thus for simplicity, we do not assign specific notations to different realizations of $\xi$}.  If a trial value $\alpha_k$ does not satisfy \eqref{eq:Armijo}, \ab{for some particular realizations of $\xi$}, then the iteration is called \emph{unsuccessful}; the new iterate is set to the previous iterate, i.e., $x_{k+1} = x_k$, and the step size parameter is set to a (fixed) fraction $\tau \leq 1$ of the previous value, i.e., $\alpha_{k+1} \gets \tau \alpha_k$. This step makes sense particularly when $g_k$ (and thus $d_k$) are random vectors  and thus can be different even for the same $x_k$. If the trial value satisfies \eqref{eq:Armijo}, then the iteration is called \emph{successful}, the new iterate is updated based on the search direction $d_k$, i.e., $x_{k+1} = x_k + \alpha_k d_k$, and the step size parameter is set to $\alpha_{k+1} \gets \tau^{-1} \alpha_k$. Algorithm \ref{alg:ls_sub}, fully specifies a subroutine for computing the step size parameter and taking a step. Note that if $\tau = 1$, Algorithm \ref{alg:grad_approx_ls} is a constant step size parameter line search algorithm. Algorithm \ref{alg:ls_sub} receives $\epsilon_f$ as input from Algorithm \ref{alg:grad_approx_ls}. We do not specify here if Algorithm \ref{alg:grad_approx_ls} receives this quantity as input from the user or has an ability to estimate it, as it may depend on a particular case. 
\begin{algorithm}[ht]
  \caption{~\textbf{Line Search Subroutine}}
\label{alg:ls_sub}
{\bf Inputs:} Current iterate $x_k$, current gradient estimate $\ab{g_k}$, current search direction $d_k$, current step size parameter $\alpha_k$, backtracking factor $\tau \in (0,1]$, Armijo parameter $c_1 \in (0,1)$, bound on the noise $\epsilon_f$.
  \begin{algorithmic}[1]
    \For{$k=0,1,2,\dots$}
        \State  {\bf Check sufficient decrease:} \label{step:suff_dec}
        
        \hspace{0.5cm} Check if  \eqref{eq:Armijo} is satisfied.
        \State {\bf if} Condition Satisfied \textbf{(\emph{successful} step)} {\bf then}
        
        \hspace{0.5cm} $x_{k+1} = x_k + \alpha_k d_k$ and $\alpha_{k+1} \gets \tau^{-1} \alpha_k$.
        \State {\bf else} 
        
        \hspace{0.5cm} $x_{k+1} = x_k$ and $\alpha_{k+1} \gets \tau \alpha_k$.
        \State {\bf Outputs:} New iterate  $x_{k+1}$, new step size parameter $\alpha_{k+1}$
    \EndFor
  \end{algorithmic}
\end{algorithm}

The modified Armijo condition has been used in \cite{berahas2019derivative}. The addition of the term $ 2\epsilon_f$ ensures that a step is \emph{successful} if $\alpha_k$ is small enough and $d_k^T\ab{g_k}$ is large enough. In  \cite{berahas2019derivative} the case of  functions with arbitrary but bounded noise, such as the ones considered here, were considered. However, unlike this paper the error of the gradient estimates was also assumed to be bounded by a constant, and convergence rates were derived for strongly convex objectives only. 

\section{Analysis of the Underlying Stochastic Process} \label{stoch_analysis}

In this section, we describe the general mechanism that is used to provide the theoretical results of the paper. This analysis is an extension of the analysis provided in \cite{cartis2018global} that accounts for possible noise in the function evaluations, i.e., $\ab{e}(x) \neq 0$. 

We begin by introducing several definitions, key assumptions and theoretical results, similar to those in \cite{cartis2018global} but suitably modified as required for the analysis in this paper. \ab{In particular, similar to \cite{cartis2018global}, we view Algorithm \ref{alg:grad_approx_ls} as a stochastic process, generated from a sequence of random function realizations $f(x_k,\xi)$ and gradient estimates $G_k$. With some abuse of notation and for simplicity of the presentation, we introduce the new probability space $(\Omega, \mathcal{F}, P)$ which includes  the randomness in both the function and the gradient realizations. Since the function realizations used by the line search are essentially replaced by their upper and lower bounds in our analysis, the nature of $\Xi$ has no affect on it.}

 The following quantities are random and are important in the analysis:  the gradient estimate $G_k$,  the step size parameter  $\Alpha_k$  and the search direction  $\mathcal{D}_k$.  Realizations of these random quantities are denoted by  $g_k = G_k(\omega_k)$, $x_k = X_k(\omega_k)$, $\alpha_k = \Alpha_k(\ab{\omega_k})$ and $d_k = \mathcal{D}_k(\omega_k)$, respectively. For brevity, we will omit the $\omega_k$ in the notation.  The iterate $X_k$\ab{, given $X_{k-1}$ and $\Alpha_{k-1}$,} is fully determined by $G_{k-1}$ and  the  noise in the function value estimation during iteration $k-1$. The noise may be stochastic or deterministic, let $\mathcal{E}_{k-1}$ denote all noise history up to iteration $k-1$. 
 Note that our algorithm and its analysis are independent of the nature of the noise, but we include $\mathcal{E}_{k-1}$ in the algorithm history for completeness. 
 We use  $\mathcal{F}_{k-1}^{G,\mathcal{E}} = \sigma(G_0,\ldots,G_{k-1},\mathcal{E}_{k-1})$ to denote the  $\sigma$-algebra generated by $G_0,\ldots,G_{k-1}$ and  $\mathcal{E}_{k-1}$,  that is to say, generated by Algorithm \ref{alg:grad_approx_ls} up to the start of iteration $k$.

\paragraph{Sufficiently accurate gradients} We assume that the random gradient approximations $G_k$  satisfy some notion of \emph{good quality}  with probability $1-\delta$. We use the following general notion of \emph{sufficiently accurate} gradients, similar to that presented in \cite{cartis2018global}. 

\begin{definition} \label{def:suff_acc_general} A sequence of random gradients $\{G_k\}$ is $(1-\delta)$-probabilistically ``sufficiently accurate'' for Algorithm \ref{alg:grad_approx_ls}, if the indicator variables
\begin{align*}
	I_k = \mathbbm{1} \{G_k \text{ is a sufficiently accurate  gradient of }\phi \text{ for the given  }\Alpha_k, X_k {\text \ and\ }  \mathcal{D}_k\}
\end{align*}
satisfy the following submartingale condition
\begin{align}	\label{prob-delta-def}
	\mathbb{P}(I_k = 1 | \mathcal{F}_{k-1}^{G,\mathcal{E}}) \geq 1 - \delta,
\end{align}
\ab{for all realizations of $\mathcal{F}_{k-1}^{G,\mathcal{E}}$, }where  $\mathcal{F}_{k-1}^{G,\mathcal{E}} = \sigma(G_0,\ldots,G_{k-1},\mathcal{E}_{k-1})$ is the $\sigma$-algebra generated by $G_0,\ldots,G_{k-1}$ and $\mathcal{E}_{k-1}$. Moreover, we say that iteration $k$ is a \textbf{true} iteration if the event $I_k = 1$ occurs, otherwise the iteration is called \textbf{false}.
\end{definition}

Definition \ref{def:suff_acc_general} is generic, but somewhat less so than the equivalent definition in \cite[Definition 2.1]{cartis2018global} where second order models are also considered and as a result the definition of \emph{``sufficient accuracy''} is not restricted to gradients. The reason Definition \ref{def:suff_acc_general} is generic is because it can be particularized differently depending on the  way the gradient estimates are generated. \ab{Specifically, in Section \ref{convergence_analysis} we define \emph{sufficiently accurate} in two different ways and derive expected complexity bounds for Algorithm \ref{alg:grad_approx_ls}. The first definition is motivated by the specific setting where estimates $g_k$ are computed via finite differences, interpolation or smoothing \cite{berahas2019derivative,berahas2019theoretical,nesterov2017random,fazel2018global,ConnScheVice08c}. The second definition is similar to that presented in \cite{cartis2018global}.}


\paragraph{Number of iterations $N_\epsilon$ to reach $\epsilon$ accuracy} The main goal of this section is to derive bounds on the expected number of iterations $\mathbb{E}[N_\epsilon]$ required to reach a desired level of accuracy $\epsilon$. \change{We formally define $N_\epsilon$ as follows.
\begin{definition}\label{def.stop}
\ 
\begin{itemize}
	\item If $\phi$ is convex or strongly convex: $N_\epsilon$ is the number of iterations required until $\phi(X_k) - \phi^\star\ \leq \epsilon$ occurs for the first time. Note, $\phi^\star = \phi(x^\star)$, where $x^\star$ is a global minimizer of $\phi$.
	\item If $\phi$ is nonconvex: $N_\epsilon$ is  the number of iterations required until $\| \nabla \phi(X_k)\| \leq \epsilon$ occurs for the first time. 
\end{itemize}
\end{definition}
}

Thus $N_\epsilon$ is a random variable with the property $\sigma(\mathbbm{1} \{N_\epsilon>k\})\subset\mathcal{F}_{k-1}^{G, \mathcal{E}} $, thus it is  a stopping time for our stochastic process; see \cite[Section 2]{cartis2018global}. To bound $\EE[N_\epsilon]$ we assume that while $k<N_\epsilon$ the  stochastic process induced  by Algorithm \ref{alg:grad_approx_ls} behaves in a certain way. Specifically, it tends to make a certain amount of progress towards optimality.

\paragraph{Measure of progress towards optimality and upper bound} As is done in \cite[Section 2]{cartis2018global}, let $Z_k$ denote a measure of progress towards optimality \ab{(from any starting point $x_0 \in \mathbb{R}^n$)}, and let $Z_\epsilon$ be an upper bound for $Z_k$, for $k<N_\epsilon$\footnote{$F_k$ and $F_\epsilon$ is the notation used in \cite{cartis2018global}.}. In particular, our analysis will use the definitions of   $Z_k$ and  $Z_\epsilon$ as described in Table \ref{tbl:prog_upper}. 
\begin{table}[h!]
\caption{ Definitions of $Z_k$ and $Z_\epsilon$ for convex, strongly convex and nonconvex functions.}
\label{tbl:prog_upper}
\centering 
\begin{tabular}{lcc}
\toprule
\textbf{Function} &
 \textbf{$\pmb{Z_k}$} &
 \textbf{$\pmb{Z_\epsilon}$}  \\  \midrule
\textbf{convex} &  $\frac{1}{\phi(X_k) - \phi^\star} - \frac{1}{\phi(X_0) - \phi^\star}$ & $\frac{1}{\epsilon} - \frac{1}{\phi(X_0) - \phi^\star}$ \\ \hdashline
\textbf{strongly convex} & $\log \left( \frac{\phi(X_0) - \phi^\star}{\phi(X_k) - \phi^\star} \right)$ & $ \log\left( \frac{\phi(X_0) - \phi^\star}{\epsilon} \right)$  \\ \hdashline
\textbf{nonconvex} &  $\phi(X_0) - \phi(X_k)$ & $\phi(X_0) -  \hat{\phi} $  \\ 
 \bottomrule
\end{tabular}
\end{table}

\bigskip

We are now ready to introduce the key assumption of the behavior of the stochastic process  $\{\Alpha_k, Z_k\}$ generated by Algorithm \ref{alg:grad_approx_ls} under which we derive a bound on $\EE[N_\epsilon]$.  In Section \ref{convergence_analysis}, we show that this assumption holds for our generic line search algorithm, under a particular definition  of {\em sufficiently accurate} gradient estimates, and thus we will be able to  derive the expected complexity bound. 

Recall that when the gradient estimate $g_k$ is  \emph{sufficiently accurate}, the iteration is called \emph{true}, and this is assumed to happen with probability at least $1-\delta$, conditioned on the past. The following assumption is a modification of the assumption in \cite[Section 2.4, Assumption 2.1]{cartis2018global}. Let $z_k = Z_k(\omega_k)$ be a realization of the random quantity $Z_k$. Note,  $z_k = Z_k(\omega_k)$ is a measure of progress towards optimality.

\begin{assumption}\label{ass:alg_behave}
There exist a constant $\bar{\alpha}>0$, a nondecreasing function $h(\alpha): \mathbb{R} \rightarrow \mathbb{R}$, which satisfies $h( \alpha)>0$ for any $\alpha>0$, and a nondecreasing function $r(\epsilon_f):\mathbb{R} \rightarrow \mathbb{R}$, which satisfies $r(\epsilon_f)\geq0$ for any $\epsilon_f\geq0$, such that for any realization of  Algorithm \ref{alg:grad_approx_ls} the following hold for all $k<N_\epsilon$:
\begin{itemize}
\item[(i)]   If iteration $k$ is \textbf{true} (i.e.,  $I_k=1$) and \textbf{successful}, then $z_{k+1}\geq z_k+h(\alpha_k)-r(\epsilon_f)$. 
\item[(ii)] If $\alpha_k  \leq  \bar{\alpha}$ and iteration $k$ is \textbf{true}  then
  iteration $k$ is also \textbf{successful}, which implies
$\alpha_{k+1}=\tau^{-1}\alpha_k$. 
\item[(iii)] $z_{k+1}\geq z_k-r(\epsilon_f)$ for all \textbf{successful} iterations $k$ and $z_{k+1}\geq z_k$ for all \textbf{unsuccessful} iteration $k$.
\item[(iv)]  The ratio $ r(\epsilon_f) / h(\bar{\alpha})$ is bounded from above by some $\gamma\in (0, 1)$.  
\end{itemize}
\end{assumption}

Assumption \ref{ass:alg_behave} provides guarantees of progress for the process $Z_k$, using  guaranteed increase $h(\alpha_k)$ and possible decrease $r(\epsilon_f)$. These quantities will be specified for each case (convex, strongly convex, nonconvex) in Section \ref{convergence_analysis}.   The key difference between Assumption \ref{ass:alg_behave} and the corresponding assumption in \cite[Section 2.4, Assumption 2.1]{cartis2018global} is that on each \emph{successful} iteration  $Z_k$ may decrease by up to $r(\epsilon_f)$. When $r(\epsilon_f)=0$,  Assumption \ref{ass:alg_behave} reduces to  the assumption in  \cite[Section 2.4, Assumption 2.1]{cartis2018global} and in this case $\gamma$ can be chosen arbitrarily close to $0$. When $r(\epsilon_f)>0$,  the  process $Z_k$ may decrease on some \emph{successful} iterations; see Assumption \ref{ass:alg_behave}(iii). Assumption \ref{ass:alg_behave}(i) states that $Z_k$ is guaranteed to increase on  \emph{true} \emph{successful} iterations by at least the quantity $h(\alpha_k)-r(\epsilon_f)$ which is positive due to Assumption \ref{ass:alg_behave}(iv). The constant $\gamma$ serves as a parameter that measures how much $h(\alpha_k)$ dominates $r(\epsilon_f)$. As we will see, $\gamma$ can be chosen to be fixed, for example $\gamma=\frac{1}{2}$, and  Assumption \ref{ass:alg_behave}(iv) then simply dictates that  $h(\alpha_k)\geq 2r(\epsilon_f)$. The guaranteed value of progress $h(\alpha_k)$ is larger when the target accuracy $\epsilon$ is larger, which in turn  implies the connection between the level of noise $\epsilon_f$ and the target accuracy $\epsilon$. In other words,  $\gamma$ is not an algorithmic parameter, it is simply a parameter whose value implies a particular bound on the neighborhood of convergence.   
 
As in  \cite{cartis2018global} we  define the following additional indicator random variables:
\begin{align*}
\Lambda_k=\mathbbm{1}\{ \Alpha_k >\bar{\alpha}\}, \qquad
\bar \Lambda_k=\mathbbm{1}\{ \Alpha_k \geq \bar{\alpha}\}, 
\end{align*}
\begin{align*}
\Theta_k=\mathbbm{1}\{{\rm Iteration\ } k\ {\rm is\ \emph{successful}\ i.e.,\ } \Alpha_{k+1}=\tau^{-1}\Alpha_k \}.
\end{align*}
Note that $\sigma(\Lambda_k)\subset \mathcal{F}_{k-1}^{G, \mathcal{E}} $, $\sigma(\bar \Lambda_k)\subset \mathcal{F}_{k-1}^{G, \mathcal{E}} $ and  $\sigma(\Theta_k)\subset \mathcal{F}_{k}^{G, \mathcal{E}} $, that is the random variables 
$\Lambda_k$ and $\bar \Lambda_k$ are fully determined by the first $k-1$ steps of the algorithm, while $\Theta_k$ is fully determined by the first $k$ steps. 

 Without loss of generality, we assume that $\bar\alpha=\tau^{c}\alpha_0$ for some positive integer $c$. In other words, $\bar\alpha$ is the largest value that the step size $\Alpha_k$ actually achieves for which  part $(ii)$ of Assumption \ref{ass:alg_behave} holds. Note that if $\tau=1$, the algorithm uses a constant step size and hence has to start with the value for which Assumption \ref{ass:alg_behave} holds, \ab{i.e., $\alpha \leq \bar\alpha$,} in order to converge. 

In summary, under Assumption \ref{ass:alg_behave}, recalling the update rules for
$\alpha_k$ in Algorithm \ref{alg:grad_approx_ls}    we can write the stochastic process $\{\Alpha_k, Z_k\}$ as obeying the expressions below:

\begin{equation}\label{eq:proc1_Zk}
\Alpha_{k+1}
= \left \{ \begin{array}{ll} \tau^{-1} \Alpha_k &{\rm if\  }\Theta_k = 1, \\
 \tau \Alpha_k &    {\rm if\  }  \Theta_k=0,
\end{array}\right . 
=\left \{ \begin{array}{ll} \tau^{-1} \Alpha_k& 
{\rm if\  } I_k=1\ {\rm and \ } \Lambda_k=0,\\ 
\tau^{-1} \Alpha_k & {\rm if\  }\Theta_k = 1, \ I_k=0\ {\rm and \ } \Lambda_k=0, \\
\tau \Alpha_k & {\rm if\  }\Theta_k = 0, \ I_k=0\ {\rm and \ } \Lambda_k=0, \\
 \tau^{-1} \Alpha_k & {\rm if\  }  \Theta_k=1\ {\rm and \ } \Lambda_k=1, \\
 \tau \Alpha_k &    {\rm if\  }  \Theta_k=0\ {\rm and \ } \Lambda_k=1,
\end{array}\right . 
\end{equation}

 \begin{equation}\label{eq:proc1_Yk}
Z_{k+1}\geq \left \{ \begin{array}{ll} Z_k+h(\Alpha_k) -r(\epsilon_f)& {\rm if\  } \Theta_k=1\ {\rm and \ } I_k=1,\\
Z_k -r(\epsilon_f)& {\rm if\  }  \Theta_k=1 \  {\rm and \ } I_k=0,\\ 
 Z_k & {\rm if\  } \Theta_k=0.
 \end{array}\right .
\end{equation}

\subsection{Analysis of the  stochastic processes}\label{sec:akfkproc}
 We now present the derivation of the bounds on  $\EE\left [N_\epsilon\right ]$  under Assumption \ref{ass:alg_behave}, 
 by modifying the analysis in  \cite{cartis2018global}. We start by introducing a useful lemma from  \cite{cartis2018global}.

\begin{lemma}\label{lem:frac_of_true}
Let $N_\epsilon$ denote the stopping time. For all $k<N_\epsilon$,
 let $I_k$ be the sequence of random variables in Definition \ref{def:suff_acc_general} so that \eqref{prob-delta-def} holds.
Let $W_k$ be a nonnegative stochastic process such that $\sigma(W_k)\subset {\cal F}^{G,\mathcal{E}}_{k-1}$, 
for any $k\geq 0$. Then,
\begin{align*}
\EE \left [\sumn  W_kI_k  \right ]\geq (1-\delta) \EE \left [\sumn W_k\right].
\end{align*}
Similarly,
\begin{align*}
\EE\left [\sumn  W_k(1-I_k )\right ]\leq \delta \EE \left [\sumn W_k\right ].
\end{align*}
\end{lemma}

\ab{For brevity, we omit the proof of Lemma \ref{lem:frac_of_true}; see \cite[Lemma 2.3]{cartis2018global}.}

The following lemma from \cite{cartis2018global} bounds the  number of steps for which $\alpha_k\leq \bar \alpha$. The proof depends only on the probabilities of different outcomes and not on the changes in $Z_k$, thus the proof from \cite{cartis2018global} applies directly. 

 \begin{lemma}\label{lem:hittime1} The expected number of steps for which $\alpha_k\leq \bar \alpha$ can be bounded as,
\begin{align*}
  \EE\left [\sumn (1-\Lambda_k)\right]\leq \frac{1}{2(1-\delta)}\EE[N_\epsilon].
\end{align*}
\end{lemma}
\begin{proof}{ 
The proof uses Lemma \ref{lem:frac_of_true} with $W_k = 1 - \Lambda_k$, and is the same as in \cite{cartis2018global}.}
\end{proof}
   
We now turn to the derivation of  the bound on 
\begin{align*}
	\EE \left [\sumn  \Lambda_k\right ],
\end{align*} 
which requires a substantially more elaborate analysis than that in \cite{cartis2018global} but is similar in spirit. The key difference is that, while in  \cite{cartis2018global} $Z_k$ never decreases, here we have to account for all iterations where $Z_k$ may decrease, and bound their expected number.   For brevity of notation, we define the following quantities:
\begin{itemize}
\item $N_{FS}=\sumn \bar \Lambda_k (1-I_k)\Theta_k$ - the number of \emph{false successful} iterations with $\Alpha_k\geq \bar\alpha$.
\item $N_{TS}=\sumn \bar \Lambda_k I_k \Theta_k$ - the number of \emph{true successful} iterations with $\Alpha_k\geq \bar\alpha$.
\item $N_{F}=\sumn \bar \Lambda_k (1-I_k)$ - the number of \emph{false}  iterations with $\Alpha_k\geq \bar\alpha$.
\item $N_{T}=\sumn \bar \Lambda_k I_k$ - the number of \emph{true} iterations with $\Alpha_k\geq \bar\alpha$.
\item $N_{TU}=\sumn  \Lambda_k I_k (1-\Theta_k)$ -  the number of \emph{true unsuccessful} iterations with $\Alpha_k>\bar\alpha$.
\item $N_{U}=\sumn  \Lambda_k  (1-\Theta_k)$ -  the number of \emph{unsuccessful} iterations with $\Alpha_k> \bar\alpha$.
\item $N_{SS}=\sumn  (1-\bar \Lambda_k)\Theta_k$ -  the number of  \emph{successful} iterations with $\Alpha_k<\bar\alpha$ (\emph{small} $\Alpha_k$).
\end{itemize}

Since 
\begin{align}\label{eq:false_true}
\EE\left [\sumn \Lambda_k\right ]= \EE \left[ \sumn \Lambda_k (1-I_k)\right ]+\EE \left[ \sumn \Lambda_k I_k\right ]\leq \EE[N_{F}]+\EE[N_{T}],
\end{align}
our goal is to bound $\EE[N_{F}]+\EE[N_{T}]$. 

We now establish several inequalities relating  the quantities we just defined. We begin with,
\begin{equation}\label{eq:M2}
N_{T}=N_{TS}+N_{TU}\leq N_{TS}+N_{U}.
\end{equation}
The equality above holds because by Assumption \ref{ass:alg_behave}(ii) there are no  \emph{true unsuccessful} iterations when $\Alpha_k= \bar\alpha$.

\begin{lemma}\label{lem:bound_on_big}
For any $l\in \{0,\ldots,N_{\epsilon}-1\}$ and for all realizations of Algorithm \ref{alg:grad_approx_ls}, we have 
\begin{align*}
\sum_{k=0}^l \Lambda_k(1- \Theta_k) \leq \sum_{k=0}^l\bar \Lambda_k\Theta_k + \log_{\tau}\left (\frac{\bar \alpha}{\alpha_0}\right ), 
\end{align*}
hence when $l=N_\epsilon-1$, 
\begin{equation}\label{eq:inc_dec}
N_{T}\leq N_{FS}+2N_{TS} + \log_{\tau}\left (\frac{\bar \alpha}{\alpha_0}\right ).
\end{equation}
\end{lemma}
\begin{proof}{On \emph{successful} iterations $\Alpha_k$ is increased and on \emph{unsuccessful} iterations $\Alpha_k$ is decreased. Hence, the total number of steps
when $\Alpha_k>\bar{\alpha}$ and $\Alpha_k$ is decreased, is bounded by the total number of steps when $\Alpha_k\geq \bar{\alpha}$ is increased plus  the number of steps required to reduce $\Alpha_k$ from its initial value $\alpha_0$ to $\bar{\alpha}$. The first inequality of the lemma is a simple consequence of this observation.

Now for $l=N_\epsilon-1$ this inequality becomes
\begin{align*}
N_U\leq N_{TS}+N_{FS} + \log_{\tau}\left (\frac{\bar \alpha}{\alpha_0}\right )
\end{align*}
which   combined with \eqref{eq:M2} gives us  \eqref{eq:inc_dec}.}
\end{proof}

\begin{lemma}\label{no_of_false_lemma} The expected number of \textbf{false}  iterations with $\Alpha_k\geq \bar\alpha$ can be bounded as,
\begin{equation*}
\EE[N_{F}]\leq \frac {\delta}{1-\delta}\EE[N_{T}]. 
\end{equation*}
\end{lemma}
\begin{proof}{ 
The proof uses Lemma \ref{lem:frac_of_true} and is the same as in \cite{cartis2018global}.}
\end{proof}

Hence, by \eqref{eq:M2} and Lemmas \ref{lem:bound_on_big} and \ref{no_of_false_lemma}, we have 
\begin{align}\label{eq:prelimbound}
	\EE[N_{F}]+\EE[N_{T}] & \leq \frac{1}{1-\delta}\EE[N_{T}] \nonumber\\
	&\leq \frac{1}{1-\delta} \left( \EE[N_{TS}]+\EE[N_{U}] \right) \nonumber\\
&\leq \frac{1}{1-\delta} \left (\EE[N_{FS}]+2\EE[N_{TS}] + \log_\tau \(\frac{\bar{\alpha}}{\alpha_0}\)\right ).
\end{align}

We now bound $\EE[N_{SS}]$, the number of  \emph{successful} iterations with $\Alpha_k<\bar\alpha$.
  \begin{lemma}\label{lem:smlsuccbnd} The expected number of  \textbf{successful} iterations with $\Alpha_k<\bar\alpha$ can be bounded as,
\begin{align*}
  \EE\left [N_{SS}\right]=  \EE\left [\sumn (1-\bar\Lambda_k)\Theta_k\right] \leq \frac{\delta}{2(1-\delta)}\EE[N_\epsilon]
\end{align*}
  \end{lemma}
\begin{proof}{ 
We want to bound the expected number of \emph{successful} iterations for which $\alpha_k<\bar\alpha$. Since on all \emph{successful} iterations $\alpha_k$ is increased, and $\alpha_0\geq \bar\alpha$, then for each such \emph{successful} iteration there has to be an \emph{unsuccessful} iteration with $\alpha_k\leq \bar \alpha$. Hence,
\begin{align*}
 \sumn (1-\bar\Lambda_k)\Theta_k\leq \sumn (1-\Lambda_k)(1-\Theta_k)\leq \sumn (1-\Lambda_k)(1-I_k).
\end{align*}
 The last inequality follows from the fact that when $\alpha_k\leq \bar\alpha$, all \emph{true} iterations are \emph{successful}, which implies $(1-\Lambda_k)I_k\leq(1-\Lambda_k) \Theta_k$.
 Now applying Lemmas \ref{lem:frac_of_true} and \ref{lem:hittime1} we have 
\begin{align*}
 \EE\left [  \sumn (1-\Lambda_k)(1-I_k)\right]\leq  \delta \EE\left [ \sumn (1-\Lambda_k)\right ]\leq \frac{\delta}{2(1-\delta)}\EE[N_\epsilon],
\end{align*}
 from which the result follows. }
\end{proof}
 
 Our next observation is  central to our analysis. It reflects the fact that the total gain minus the total loss in $Z_k$ is bounded from above by $Z_\epsilon$.
We observe that when $\Alpha_k\geq \bar{\alpha}$ on \emph{true successful} iterations  this gain is bounded from below away from zero by 
$h(\bar \alpha)-r(\epsilon_f)\geq (1-\gamma)h(\bar \alpha)$ and at other \emph{successful} iterations the loss  is bounded above by $r(\epsilon_f)$. This will allow us to bound $\EE[N_{TS}]$. 

\begin{lemma}\label{lem:nN2}
The number of \textbf{true successful} iterations with $\Alpha_k\geq \bar\alpha$ can be bounded as,
\begin{align}\label{eq:N2_2}
	N_{TS}\leq  \frac{Z_\epsilon}{(1-\gamma)h(\bar{\alpha})} +  \frac{\gamma}{1-\gamma}(N_{FS}+N_{SS})
\end{align}
and, hence, 
\begin{align}\label{eq:EN2_2}
	\EE\left[N_{TS}\right] \leq  \frac{Z_\epsilon}{(1-\gamma)h(\bar{\alpha})} +  \frac{\gamma}{1-\gamma}\EE\left[N_{FS}\right]+\frac{\gamma}{1-\gamma}\frac{\delta}{2(1-\delta)}\EE[N_\epsilon].
\end{align}
\end{lemma}
\begin{proof}{The proof follows directly from \eqref{eq:proc1_Yk} and Assumption \ref{ass:alg_behave}. $Z_k$ is increased by at least 
$h(\bar{\alpha})-r(\epsilon_f)$ at each 
\emph{true successful} iteration when $\alpha_k\geq \bar \alpha$ and it may be decreased at most $r(\epsilon_f)$ at each \emph{false successful} iteration when  $\alpha_k\geq \bar \alpha$ 
and at each \emph{successful} iteration when $\alpha_k<\bar \alpha$. Thus, we have
\begin{equation*}
Z_\epsilon \geq Z_k \geq N_{TS}{(h(\bar{\alpha})-r(\epsilon_f))} -  r(\epsilon_f)(N_{FS}+N_{SS}).
\end{equation*}
Recalling that by  Assumption \ref{ass:alg_behave}, $r(\epsilon_f) \leq \gamma h(\bar{\alpha})$ and $\gamma\in (0,1)$ we obtain \eqref{eq:N2_2}, while \eqref{eq:EN2_2} follows further from Lemma \ref{lem:smlsuccbnd}. 
}
\end{proof}

\begin{lemma}\label{lem:nN3}
Under the condition that $\delta< \frac{1}{2}-\frac{\gamma}{2}$, the number of \textbf{false successful} iterations with $\Alpha_k\geq \bar\alpha$ can be bounded as,
\begin{align*}
	\EE\left [N_{FS}\right ]\leq \frac{2\delta}{1-2\delta-\gamma} \frac{Z_\epsilon}{h(\bar{\alpha})} 
+ \frac{(1-\gamma)\delta}{1-2\delta-\gamma}\log_\tau\left(\frac{\bar{\alpha}}{\alpha_0}\right) +   \frac{\delta^2\gamma}{(1-\delta)(1-2\delta-\gamma)}  \EE[N_\epsilon].
\end{align*}
\end{lemma}
\begin{proof}{
From \eqref{eq:inc_dec} and Lemma \ref{no_of_false_lemma}, we have
\begin{align*}
	\EE\left [N_{FS}\right ]\leq  \EE\left [N_{F}\right ]\leq \frac{\delta}{1-\delta}\left[ \EE\left[N_{FS}\right]+2\EE\left[N_{TS}\right] + \log_\tau\left(\frac{\bar{\alpha}}{\alpha_0}\right)\right ].
\end{align*}
Then, from Lemma \ref{lem:nN2} if follows that
\begin{align*}
	\EE\left[N_{FS}\right]\leq  \frac{\delta}{1-\delta}\left[  \frac{1+\gamma}{1-\gamma}\EE\left[N_{FS}\right]+\frac{2Z_\epsilon}{(1-\gamma)h(\bar{\alpha})}  +\frac{\gamma}{1-\gamma}\frac{\delta}{1-\delta}\EE[N_\epsilon]+ \log_\tau\left(\frac{\bar{\alpha}}{\alpha_0}\right)\right ].
\end{align*}
Collecting the terms involving $\EE\left[N_{FS}\right]$ on the left and observing that $1-\frac{1+\gamma}{1-\gamma}\frac{\delta}{1-\delta}=\frac{1-2\delta-\gamma}{(1-\gamma)(1-\delta)}$ we can derive the bound
\begin{align*}
	\EE\left[N_{FS}\right ]\leq  \frac{(1-\gamma)\delta}{1-2\delta-\gamma}\left[ \frac{2Z_\epsilon}{(1-\gamma)h(\bar{\alpha})}  +\frac{\gamma}{1-\gamma}\frac{\delta}{1-\delta}\EE[N_\epsilon]+ \log_\tau\left(\frac{\bar{\alpha}}{\alpha_0}\right)\right ],
\end{align*}
from which the result follows.
}
\end{proof}
We can now derive the bound for $\EE\left [N_{TS}\right ]$  using Lemmas \ref{lem:nN2} and \ref{lem:nN3} and collecting the appropriate terms. 
\begin{lemma}\label{lem:nN4} Under the condition that $\delta< \frac{1}{2}-\frac{\gamma}{2}$, the number of \textbf{true successful} iterations with $\Alpha_k\geq \bar\alpha$ can be bounded as,
\begin{align*}
	\EE\left [N_{TS}\right]\leq    \frac{1-2\delta}{1-2\delta-\gamma}\frac{Z_\epsilon}{h(\bar{\alpha})} + \frac{\gamma \delta}{1-2\delta-\gamma}\log_\tau\left(\frac{\bar{\alpha}}{\alpha_0}\right) +  \frac{\gamma(1-2\delta)\delta}{2(1-\delta)(1-2\delta-\gamma)}  \EE[N_\epsilon]. 
\end{align*}
\end{lemma}
\begin{proof}{ From Lemma \ref{lem:nN2}, we have
\begin{equation*}
\EE\left[N_{TS}\right] \leq  \frac{Z_\epsilon}{(1-\gamma)h(\bar{\alpha})} +  \frac{\gamma}{1-\gamma}\EE\left[N_{FS}\right]+\frac{\gamma}{1-\gamma}\frac{\delta}{2(1-\delta)}\EE[N_\epsilon].
\end{equation*}
Using the result from Lemma \ref{lem:nN3}, it follows that
\begin{align*}
	&\ \EE\left[N_{TS}\right] \\
	 \leq & \frac{Z_\epsilon}{(1-\gamma)h(\bar{\alpha})} \\
	& \ +  \frac{\gamma}{1-\gamma}\left[ \frac{2\delta}{1-2\delta-\gamma} \frac{Z_\epsilon}{h(\bar{\alpha})} 
+ \frac{(1-\gamma)\delta}{1-2\delta-\gamma}\log_\tau\left(\frac{\bar{\alpha}}{\alpha_0}\right) +   \frac{\delta^2\gamma}{(1-\delta)(1-2\delta-\gamma)}  \EE[N_\epsilon]\right] \\
	  & \ +\frac{\gamma}{1-\gamma}\frac{\delta}{2(1-\delta)}\EE[N_\epsilon]\\
	 = & \frac{1-2\delta}{1-2\delta-\gamma}\frac{Z_\epsilon}{h(\bar{\alpha})} + \frac{\gamma\delta}{1-2\delta-\gamma}\log_\tau\left(\frac{\bar{\alpha}}{\alpha_0}\right) +  \frac{\gamma(1-2\delta)\delta}{2(1-\delta)(1-2\delta-\gamma)}  \EE[N_\epsilon],
\end{align*}
which completes the proof.
}
\end{proof}

\begin{lemma}\label{lem:hittime2}
Under the condition that  $\delta< \frac{1}{2}-\frac{\gamma}{2}$, the number of  iterations with $\Alpha_k > \bar\alpha$ can be bounded as,
\begin{align*}
	\EE\left [\sumn \Lambda_k \right ] 
	\leq &  \frac{2}{1-2\delta-\gamma} \frac{Z_\epsilon}{h(\bar{\alpha})} + \frac{(1-\gamma)}{1-2\delta-\gamma}\log_\tau\left(\frac{\bar{\alpha}}{\alpha_0}\right)+ \frac{\gamma \delta}{(1-\delta)(1-2\delta-\gamma)}\EE[N_\epsilon].
\end{align*}
\end{lemma}
\begin{proof}{
By \eqref{eq:false_true}, \eqref{eq:prelimbound} and Lemmas \ref{lem:nN3} and \ref{lem:nN4}, we have 
\begin{align*}
	&\ \EE\left [\sumn \Lambda_k\right ] \leq\EE[N_{F}]+\EE[N_{T}]\\
	\leq & \frac{1}{1- \delta}\left[\EE[N_{FS}]+2\EE[N_{TS}] + \log_\tau(\bar{\alpha}/\alpha_0)\right] \nonumber\\
	\leq &  \frac{1}{1-\delta} \left[ \frac{2\delta}{1-2\delta-\gamma} \frac{Z_\epsilon}{h(\bar{\alpha})} 
+ \frac{(1-\gamma)\delta}{1-2\delta-\gamma}\log_\tau\left(\frac{\bar{\alpha}}{\alpha_0}\right) +   \frac{\delta^2\gamma}{(1-\delta)(1-2\delta-\gamma)}  \EE[N_\epsilon] \right]\nonumber \\
& + \frac{2}{1-\delta} \left[ \frac{1-2\delta}{1-2\delta-\gamma}\frac{Z_\epsilon}{h(\bar{\alpha})} + \frac{\gamma \delta}{1-2\delta-\gamma}\log_\tau\left(\frac{\bar{\alpha}}{\alpha_0}\right) +  \frac{\gamma(1-2\delta)\delta}{2(1-\delta)(1-2\delta-\gamma)}  \EE[N_\epsilon] \right]\nonumber \\
& + \frac{1}{1-\delta} \log_\tau\left(\frac{\bar{\alpha}}{\alpha_0}\right)\nonumber \\
 = &\frac{2}{1-2\delta-\gamma} \frac{Z_\epsilon}{h(\bar{\alpha})} + \frac{(1-\gamma)}{1-2\delta-\gamma}\log_\tau\left(\frac{\bar{\alpha}}{\alpha_0}\right)+ \frac{\gamma \delta}{(1-\delta)(1-2\delta-\gamma)}\EE[N_\epsilon],\nonumber
\end{align*}
which completes the proof.
}
\end{proof}

Combining Lemmas \ref{lem:hittime1} and \ref{lem:hittime2}, we have the key bound
\begin{align*}
	\EE\left [ N_\epsilon\right ] \leq & \EE\left [\sumn \Lambda_k\right ]+\EE\left [\sumn (1- \Lambda_k)\right ]\\
  \leq & \frac{2}{1-2\delta-\gamma} \frac{Z_\epsilon}{h(\bar{\alpha})} + \frac{(1-\gamma)}{1-2\delta-\gamma}\log_\tau\left(\frac{\bar{\alpha}}{\alpha_0}\right) \\
  & + \frac{\gamma \delta}{(1-\delta)(1-2\delta-\gamma)}\EE[N_\epsilon] + \frac{1}{2(1-\delta)}\EE[N_\epsilon].
\end{align*}

Collecting the terms with $\EE\left [ N_\epsilon\right ]$ on the left-hand side and multiplying both sides by $1-2\delta-\gamma$ we obtain 
\begin{align*}
	\left[ 1-2\delta-\gamma - \frac{\gamma \delta}{1-\delta} - \frac{1-2\delta-\gamma}{2(1-\delta)}\right]\EE\left [ N_\epsilon\right ]\leq   \frac{2Z_\epsilon}{h(\bar{\alpha})} + (1-\gamma)\log_\tau\left(\frac{\bar{\alpha}}{\alpha_0}\right)
\end{align*}
If the coefficient in front of $\EE\left [ N_\epsilon\right ]$ is positive, that immediately gives us a bound on the expected stopping time $\EE\left [ N_\epsilon\right ]$. 
This coefficient is 
\begin{align*}
1-2\delta-\gamma - \frac{\gamma \delta}{1-\delta} - \frac{1-2\delta-\gamma}{2(1-\delta)}=\frac{4\delta^2-4\delta+1-\gamma}{2(1-\delta)} = \frac{(1 - 2\delta)^2-\gamma}{2(1-\delta)}. 
\end{align*}

The smaller of the two roots of $4\delta^2-4\delta+1-\gamma$ is $\frac{1}{2}-\frac{\sqrt{\gamma}}{2} \leq \frac{1}{2}-\frac{\gamma}{2} $. Hence, we have the following final bound.

\begin{theorem} \label{th:mainbound}
Under the condition that $\delta< \frac{1}{2}-\frac{\sqrt{\gamma}}{2}$, the stopping time $N_\epsilon$ is bounded in expectation as follows
\begin{align}	\label{eq:final_bound}
	\EE [N_\epsilon] \leq \frac{2(1-\delta)}{(1 - 2\delta)^2-\gamma}\left [\frac{2Z_\epsilon}{h(\bar{\alpha})} +(1-\gamma)\log_\tau\left (\frac{\bar{\alpha}}{\alpha_0}\right)\right]
\end{align}
\end{theorem}

\begin{remark} 
The result of Theorem \ref{th:mainbound} is a generalization of the result in \cite{cartis2018global} to the case where the function is computed with some noise. Specifically, when $\epsilon_f = 0$, and as a result $r(\epsilon_f) = 0$, then $\gamma = 0$ and \eqref{eq:final_bound} reduces to the bounds in \cite{cartis2018global}. We should note that the condition $\delta <\frac{1}{2}$ corresponds to the condition $p>\frac{1}{2}$ in \cite{cartis2018global}. If, on the other hand $\delta=0$, then we recover the deterministic complexity bound. If  $r(\epsilon_f) > 0$ the quantity $\gamma$ can be chosen to be some fixed constant, for example $\frac{1}{2}$. This implies the condition that $\delta<\frac{1}{2}(1-\frac{1}{\sqrt{2}})$ and the bound \eqref{eq:final_bound} is adjusted accordingly. We see that larger values of $\gamma$ imply tighter bounds on $\delta$, on the other hand as we will see in the next section, they allow the algorithm to achieve better accuracy for the same noise level $\epsilon_f$. Thus, the constant $\gamma$ simply serves as a means to highlight the trade-off between imposing smaller bounds on $\delta$ and achieving a smaller  radius of convergence. 
\end{remark}

\section{Convergence Analysis of the Modified Line Search} \label{convergence_analysis}

\ab{In this section, we derive expected complexity bounds for the modified line search Algorithm \ref{alg:grad_approx_ls}, where the step size parameter is chosen using Algorithm \ref{alg:ls_sub}. 

We begin by stating the first condition on the gradient estimates which we use in our analysis,
\begin{align}	\label{eq:theta_cond}
	\| \ab{g_k} - \nabla \phi(x_k) \| \leq \theta \|  \nabla \phi(x_k) \|,  \ \ \text{for all} \ \ k=0, 1, 2,\ldots,
\end{align}
for some $\theta \in [0,1)$. This condition is referred to as a \emph{norm condition} and was introduced and studied in \cite{carter1991global} in the context of trust-region methods with inaccurate gradients. Note, this condition implies that \ab{$g_k$} is a descent direction for the function $\phi$. When unbiased  stochastic estimators  of $\nabla \phi(x)$ are available,  \ab{$g_k$} can be computed by averaging these estimators.  If the variance of these estimators is bounded by ${\cal O}(\| \nabla \phi(x_k) \| ^2)$, then condition \eqref{eq:theta_cond}  can be satisfied, with probability $1-\delta$, by using  a sufficiently large number of the estimators (batch size) to compute \ab{$g_k$}.  We chose to consider condition \eqref{eq:theta_cond} because we are motivated by the specific setting where estimates $g_k$ are computed via finite differences, interpolation or smoothing \cite{berahas2019derivative,berahas2019theoretical,nesterov2017random,fazel2018global,ConnScheVice08c}. 

In a more general stochastic setting, unless one knows $\|  \nabla \phi(x_k) \|$, condition \eqref{eq:theta_cond} is hard or impossible to verify or guarantee. A simple way of making condition \eqref{eq:theta_cond} realizable is to replace  $ \| \nabla \phi(x_k) \|$ with $\epsilon$, where $\epsilon$ is the desired convergence accuracy. However, if the cost of obtaining \ab{$g_k$} that satisfies  $\| \ab{g_k} - \nabla \phi(x_k) \| \leq \theta \epsilon$ increases as $\epsilon$ decreases, replacing $\|  \nabla \phi(x_k) \|$ by its global lower bound $\epsilon$ can lead to inefficient algorithms.

There is significant amount of work that attempts to circumvent the aforementioned difficulties in the case of general stochastic gradient estimates; see e.g., \cite{byrd2012sample,cartis2018global,paquette2018stochastic}. In \cite{byrd2012sample} a practical approach to estimate $\|  \nabla \phi(x_k) \|$ is proposed  and  used to ensure some approximation of \eqref{eq:theta_cond} holds. In \cite{cartis2018global} and \cite{paquette2018stochastic}, \eqref{eq:theta_cond} is replaced with a condition that for some $\kappa \geq 0$,
\begin{align}\label{eq:alpha_cond}
	\| \ab{g_k} - \nabla \phi(x_k) \| \leq \kappa\alpha_k \|  \ab{g_k} \|,  \ \ \text{for all} \ \ k=0, 1, 2,\ldots, 
\end{align}
 holds with probability $1-\delta$ and it is discussed how this condition can be ensured. Under this condition,  expected complexity bounds  are derived for a line search method that has access to deterministic function values in \cite{cartis2018global} and stochastic function values (with additional assumptions)  in  \cite{paquette2018stochastic}.  While this condition does not require the variance to diminish with $\|  \nabla \phi(x_k) \|$, it may be hard or impossible to ensure when $\alpha_k$ is very small, due to the noise. Thus, we propose the following modification of this condition,
\begin{align}\label{eq:alpha_cond_mod}
	\| \ab{g_k} - \nabla \phi(x_k) \| \leq \max \{ {\zeta \epsilon_g},\kappa\alpha_k \|  \ab{g_k} \| \},  \ \ \text{for all} \ \ k=0, 1, 2,\ldots,
\end{align}
{where $\zeta > 1$ and $\epsilon_g \geq 0$ (we precisely define $\epsilon_g$ in Section \ref{sec:conv_alpha}.)}
We extend the analysis in \cite{cartis2018global} and derive complexity bounds based on \eqref{eq:alpha_cond_mod} for our setting (i.e., noisy function evaluations).

In the remainder of this section, we present a convergence analysis for the generic line search algorithm (Algorithms \ref{alg:grad_approx_ls}-\ref{alg:ls_sub}). The analysis is an extension of the analysis presented in \cite{cartis2018global} to the case where functions are computed with noise (Assumption \ref{assum:func_bnd}). We first consider the \emph{norm condition} \eqref{eq:theta_cond}, and prove complexity guarantees for the special case where $d_k=-g_k$ (Section \ref{sec:conv_norm}) and general descent (Section \ref{sec.gen_descent}). We then prove similar results for condition \eqref{eq:alpha_cond_mod} (Section \ref{sec:conv_alpha}). For brevity we omit the results for general descent under condition \eqref{eq:alpha_cond_mod} as these results are very similar to those for \eqref{eq:theta_cond}.}

\subsection{Convergence under Condition \eqref{eq:theta_cond}}
\label{sec:conv_norm}
We use the following notion of \emph{sufficiently accurate gradients}.
\begin{definition} \label{def:suff_acc}
A sequence of random gradients $\{G_k\}$ is $(1-\delta)$-probabilistically ``sufficiently accurate'' for Algorithm \ref{alg:grad_approx_ls} 
 if there exists a constant $\theta \in [0,\tfrac{1-c_1}{2-c_1})$, such that the indicator variables
\begin{align*}
	I_k = \mathbbm{1}\{ \| G_k - \nabla \phi(X_k) \| \leq \theta \| \nabla \phi(X_k)\|\}
\end{align*}
satisfy the following submartingale condition
\begin{align*}
	\mathbb{P}(I_k = 1 | \mathcal{F}_{k-1}^{G, \mathcal{E}} ) \geq 1 - \delta,
\end{align*}
\ab{for all realizations of $\mathcal{F}_{k-1}^{G, \mathcal{E}}$, }where $\mathcal{F}_{k-1}^{G, \mathcal{E}}  = \sigma(G_0,\ldots,G_{k-1},\mathcal{E}_{k-1})$ is the $\sigma$-algebra generated by $G_0,\ldots,G_{k-1}$ and  
$\mathcal{E}_{k-1}$. Moreover, we say that iteration $k$ is a \textbf{true} iteration if the event $I_k = 1$ occurs, otherwise the iteration is called \textbf{false}.
\end{definition}
For the remainder of this section, we make the following additional assumption.
\begin{assumption} \label{assum:accurate} \textbf{(Sufficiently accurate gradients)} The sequence of random gradients $\{G_k\}$ generated by Algorithm \ref{alg:grad_approx_ls} are $(1-\delta)$-probabilistically ``sufficiently accurate'' with $\delta < \frac{1}{2}-\frac{\sqrt{\gamma}}{2}$, for some $\gamma\in (0,1)$.
\end{assumption}


Equipped with the above definitions, assumptions and theorems, we now provide convergence guarantees for the generic line search algorithm (Algorithm \ref{alg:grad_approx_ls}-\ref{alg:ls_sub}), for convex, strongly convex and nonconvex objective functions. \change{We remind the reader of the definition of the stopping time $N_\epsilon$ given in Definition \ref{def.stop}.}


For each \emph{true} iteration (i.e., $I_k = 1$), we have
\begin{align*}	
	\| \ab{g_k} - \nabla \phi(x_k) \| \leq \theta \|  \nabla \phi(x_k) \|,
\end{align*}
which implies, using the triangle inequality that
\begin{align} \label{eq:boundgk}
	\| \ab{g_k}\| \geq (1-\theta) \| \nabla \phi(x_k)\|.
\end{align}

We now show that Assumption \ref{ass:alg_behave} is satisfied. To this end, for the three classes of functions, we show that there exists an upper bound $\bar{\alpha}$ on the step length parameter, and functions $h(\alpha)$ and $r(\epsilon_f)$ such that the assumption is true. First, we derive an expression for the constant $\bar{\alpha}$. 
 
\begin{lemma}\label{stepsize_threshold_lemma1}
Let Assumptions \ref{assum:lip_cont} and \ref{assum:func_bnd}  hold. For every realization  of Algorithm \ref{alg:grad_approx_ls},  if iteration $k$ is \textbf{true} (i.e., $I_k=1$), and if 
\begin{equation}\label{alpha_bd1}
	\alpha_k \leq \bar{\alpha} = \frac{2(1-2\theta - c_1(1-\theta))}{L(1-\theta)},
\end{equation}
then \eqref{eq:Armijo} holds. In other words, when \eqref{alpha_bd1} holds, any \textbf{true} iteration is also a \textbf{successful} iteration. Moreover, for every \textbf{true} and \textbf{successful} iteration,
\begin{align}		\label{eq.decrease_phi}
	\phi(x_{k+1}) &\leq \phi(x_k) - c_1 \alpha_k (1-\theta)^2\| \nabla \phi(x_k) \|^2 + 4\epsilon_f.
\end{align}
\end{lemma}

\begin{proof}{ By Assumption \ref{assum:lip_cont}, we have
\begin{align*}
	\phi(x_k - \alpha_k \ab{g_k}) \leq \phi(x_k) - \alpha_k \ab{g_k}^T\nabla \phi(x_k) +\frac{\alpha_k^2L}{2}\|\ab{g_k}\|^2,
\end{align*} 
Applying the Cauchy-Schwarz inequality, \eqref{eq:theta_cond} and \eqref{eq:boundgk}, for every \emph{true} iteration 
\begin{align*}
	\phi(x_k - \alpha_k \ab{g_k}) & \leq \phi(x_k) - \alpha_k \ab{g_k}^T \nabla \phi (x_k) + \frac{\alpha_k^2 L}{2} \| \ab{g_k}\|^2\\
			& = \phi(x_k) - \alpha_k \ab{g_k}^T (\nabla \phi (x_k) - \change{g_k}) - \alpha_k \left[1 -\frac{\alpha_k L}{2}\right] \| \ab{g_k}\|^2\\
			& \leq \phi(x_k) + \alpha_k \| \ab{g_k}\| \|\nabla \phi (x_k) - \ab{g_k}\| - \alpha_k \left[1 -\frac{\alpha_k L}{2}\right] \| \ab{g_k}\|^2\\
			& \leq \phi(x_k) + \frac{\alpha_k \theta}{1 - \theta }\| \ab{g_k} \|^2 - \alpha_k \left[1 -\frac{\alpha_k L}{2}\right] \| \ab{g_k} \|^2\\
			& = \phi(x_k) - \alpha_k \left[\frac{1-2\theta}{1-\theta} -\frac{\alpha_k L}{2}\right] \| \ab{g_k}\|^2.
\end{align*}
By Assumption \ref{assum:func_bnd}, we have
\begin{align*}
	f(x_k - \alpha_k \change{g_k},\ab{\xi}) \leq f(x_k,\ab{\xi}) - \alpha_k \left[\frac{1-2\theta}{1-\theta} -\frac{\alpha_k L}{2}\right] \| \ab{g_k}\|^2 + 2 \epsilon_f.
\end{align*}
From this we conclude that \eqref{eq:Armijo} holds whenever 
\begin{align*}
	f(x_k,\ab{\xi}) - \alpha_k\left[\frac{1-2\theta}{1-\theta} -\frac{\alpha_k L}{2}\right] \| \ab{g_k} \|^2 + 2 \epsilon_f \leq f(x_k,\ab{\xi}) - c_1 \alpha_k \| \ab{g_k} \|^2 + 2\epsilon_f,
\end{align*}
which is equivalent to \eqref{alpha_bd1}. Therefore, using Assumption \ref{assum:func_bnd} and \eqref{eq:boundgk}, for every \emph{true} and \emph{successful} iteration we have
\begin{align*}
	\phi(x_{k+1}) &\leq \phi(x_k) - c_1 \alpha_k (1-\theta)^2\| \nabla \phi(x_k) \|^2 + 4\epsilon_f, 
\end{align*}
which completes the proof.
}
\end{proof}

We should mention that when the error in the gradient approximation is zero, i.e., $\theta = 0$, we recover the step size parameter condition from the deterministic setting. Moreover, when there is no noise in the function, i.e., $\epsilon_f = 0$, we recover the sufficient decrease condition of the deterministic gradient descent algorithm with an Armijo backtracking line search.

Next, we state and prove a result for the case of \emph{false} and \emph{successful} iterations.

\begin{lemma}\label{stepsize_threshold_lemma1_succ}
Let Assumption \ref{assum:func_bnd}  hold. For every \textbf{false} and \textbf{successful} iteration of Algorithm \ref{alg:grad_approx_ls},  we have
\begin{align*}	
	\phi(x_{k+1}) &\leq \phi(x_k) - c_1 \alpha_k \| \ab{g_k} \|^2 + 4\epsilon_f.
\end{align*}
\end{lemma}

\begin{proof}{ The proof of this lemma is straightforward. For every \emph{successful} iteration we have
\begin{align*}
	f(x_{k+1},\ab{\xi}) &\leq f(x_k,\ab{\xi}) - c_1 \alpha_k \| \ab{g_k} \|^2 + 2\epsilon_f.
\end{align*}
Thus, by Assumption \ref{assum:func_bnd}, 
\begin{align*}	
	\phi(x_{k+1}) &\leq \phi(x_k) - c_1 \alpha_k \| \ab{g_k} \|^2 + 4\epsilon_f,
\end{align*}
which completes the proof.
}
\end{proof}

The result of Lemma \ref{stepsize_threshold_lemma1_succ} shows the amount of decrease on \emph{false} and \emph{successful} iterations. Note, the error term $4\epsilon_f$ illustrates that on \emph{false} iterations the function value may increase and that the increase is related to the noise in the function values.

\subsubsection{Convex Functions}

In this section, we analyze the expected complexity of Algorithm \ref{alg:grad_approx_ls} in the case when $\phi$ is a convex function.
\begin{assumption} \label{assum:conv} \textbf{(Convexity and boundedness of iterates)}  The function $\phi$ is convex and
  there exists a constant $D >0$ such that
  \begin{align}\label{eq:conv_assum}
 \norm{x-x^\star} \le D \quad \text{for all $x\in \mathcal{U}$,}
 \end{align}
where  $x^\star$ is some global minimizer of \ab{$\phi$} (and $\phi^\star = \phi(x^\star)$) and the set $\ \mathcal{U}$ contains all iteration realizations.
\end{assumption}

This assumption may seem strong since it requires all iterates of the algorithm to remain in a bounded region. When the objective function is not allowed to increase, this assumption is simply ensured by assuming bounded level sets of $\phi(x)$. In the case of noisy function values in principle, iterates can wander out of a bounded region with some small probability (as this will require a large sequence of \emph{false} \emph{successful} iterations). Thus, ideally, we need to modify the algorithm to prevent it from going outside of some predefined bounded region, that is known to contain $x^\star$. Such modification is simple and our analysis will still apply, but with some notational complications. Therefore, we choose not to impose this modification explicitly. Note, we only use this assumption in the convex case and drop it in the strongly convex and nonconvex cases, and, thus, the nonconvex case convergence rates apply to the convex case without 
\eqref{eq:conv_assum}. 

We bound the number of iterations taken by Algorithm \ref{alg:grad_approx_ls} until $\phi (X_k ) - \phi^\star \leq \epsilon$ occurs. Let
\begin{align}	\label{eq:Delta_k}
	\Delta_k^\phi = \phi(X_k) - \phi^\star, \quad \text{and} \quad Z_k = \frac{1}{\Delta_k^\phi}  - \frac{1}{\Delta_0^\phi} .
\end{align}
By this definition, $N_\epsilon$ is the number of iterations taken until $Z_k \geq \frac{1}{\epsilon}  - \frac{1}{\Delta_0^\phi}  = Z_\epsilon$. Note, that due to the noise in the function evaluations, $\epsilon$ cannot be chosen to be arbitrarily small. We make an assumption on $\epsilon$ that explicitly defines the neighborhood of convergence.
\begin{assumption} \label{assum:epsilon_conv} \textbf{(Neighborhood of convergence, convex case)}  
\begin{align*}
		\epsilon^2 > \max \left\{\frac{8 \epsilon_f LD^2}{\gamma c_1 (1-\theta)(1-2\theta - c_1(1-\theta))} , 16 \epsilon_f^2\right\},
\end{align*}
with the same  $\gamma\in(0,1)$ as used in Assumption  \ref{assum:accurate}. 
\end{assumption}

\begin{remark} 
We will show that the above assumption   implies Assumption \ref{ass:alg_behave}(iv) with the same 
constant $\gamma$. Hence here we see the direct connection between $\gamma$ and the lower bound on $\epsilon$. As discussed previously, $\gamma$ can be chosen to be $\frac{1}{2}$, for example. 
\end{remark}

By Lemma \ref{stepsize_threshold_lemma1}, whenever $\Alpha_k \leq \bar{\alpha}$, then every \emph{true} iteration is also \emph{successful}. We now show that on \emph{true} and \emph{successful} iterations, $Z_k$ increases by at least some function $h(\Alpha_k) - r(\epsilon_f)$, for all $k < N_\epsilon$.

\begin{lemma} \label{lemm:convex} Let Assumptions \ref{assum:func_bnd}, \ref{assum:conv} and \ref{assum:epsilon_conv} hold, and consider any realization of Algorithm \ref{alg:grad_approx_ls}. For every iteration that is \textbf{true} and \textbf{successful}, we have 
\begin{align*}	
	\ab{ z_{k+1}  \geq z_k}+ \frac{c_1 \alpha_k(1-\theta)^2 }{4D^2} - \frac{4 \epsilon_f}{\epsilon^2}.
\end{align*}
\end{lemma}
\begin{proof}{ By Assumption \ref{assum:conv}, for all $x,y \in \mathbb{R}^n$, we have
\begin{align*}
	\phi(x) - \phi(y) \geq \nabla \phi(y)^T (x-y).
\end{align*}
Thus, if $x = x^\star$ and $y = x_k$, we have
\begin{align*}
	- \Delta_k^\phi = \phi^\star - \phi(x_k) &\geq \nabla \phi(x_k)^T (x^\star-x_k) \geq -D \|\nabla \phi(x_k)\|, 
\end{align*}
where we used the Cauchy-Schwarz inequality and \eqref{eq:conv_assum}. Thus, when $k$ is a \emph{true} iteration, by \eqref{eq:boundgk} we have
\begin{align*}
	\| \ab{g_k}\|^2 \geq (1-\theta)^2 \| \nabla \phi(x_k)\|^2 \geq \frac{(1-\theta)^2 \left(\Delta_k^\phi \right)^2}{D^2}.
\end{align*}
If $k$ is also a \emph{successful} iteration, then
\begin{align*}
	\Delta_k^\phi - \Delta_{k+1}^\phi = \phi(x_k) - \phi(x_{k+1}) \geq c_1 \alpha_k \| \ab{g_k}\|^2 - 4 \epsilon_f\geq \frac{c_1 \alpha_k(1-\theta)^2 \left(\Delta_k^\phi \right)^2}{D^2} - 4 \epsilon_f,
\end{align*}
and thus, 
\begin{align*}
	\Delta_k^\phi + 4\epsilon_f - \Delta_{k+1}^\phi \geq \frac{c_1 \alpha_k(1-\theta)^2 \left(\Delta_k^\phi \right)^2}{D^2}.
\end{align*}
Dividing by $\left(\Delta_{k+1}^\phi \right)\left( \Delta_{k}^\phi + 4 \epsilon_f \right)$
\begin{align}	\label{eq:conv1}
	\frac{1}{\Delta_{k+1}^\phi} - \frac{1}{\Delta_k^\phi + 4\epsilon_f}  \geq \frac{c_1 \alpha_k(1-\theta)^2 \left(\Delta_k^\phi \right)^2}{D^2\left(\Delta_{k+1}^\phi \right)\left( \Delta_{k}^\phi + 4 \epsilon_f \right)}.
\end{align}
The left-hand-side of \eqref{eq:conv1} can be bounded by
\begin{align*}
	\frac{1}{\Delta_{k+1}^\phi} - \frac{1}{\Delta_k^\phi + 4\epsilon_f}  & = \frac{1}{\Delta_{k+1}^\phi} - \frac{1}{\Delta_{k}^\phi} + \frac{1}{\Delta_{k}^\phi} - \frac{1}{\Delta_k^\phi + 4\epsilon_f} \\
		& = \frac{1}{\Delta_{k+1}^\phi} - \frac{1}{\Delta_{k}^\phi} + \frac{4\epsilon_f}{\left(\Delta_k^\phi \right)\left( \Delta_{k}^\phi + 4 \epsilon_f \right)}\\
		& \leq \frac{1}{\Delta_{k+1}^\phi} - \frac{1}{\Delta_{k}^\phi} + \frac{4 \epsilon_f}{\epsilon^2},
\end{align*}
where the last inequality holds since $\Delta_k^\phi + 4\epsilon_f \geq \Delta_k^\phi \geq \epsilon$. The right-hand-side of \eqref{eq:conv1} can be bounded by
\begin{align*}
	\frac{c_1 \alpha_k(1-\theta)^2 \left(\Delta_k^\phi \right)^2}{D^2\left(\Delta_{k+1}^\phi \right)\left( \Delta_{k}^\phi + 4 \epsilon_f \right)}  & \geq \frac{c_1 \alpha_k(1-\theta)^2 \left(\Delta_k^\phi \right)^2}{D^2\left( \Delta_{k}^\phi + 4 \epsilon_f \right)^2} \\
				& \geq \frac{c_1 \alpha_k(1-\theta)^2 }{4D^2}
\end{align*}
where the first inequality holds since $\Delta_{k+1}^\phi  \leq  \Delta_{k}^\phi + 4 \epsilon_f $, and the second due to the fact that $\Delta_k^\phi \geq \epsilon > 4 \epsilon_f$ (due to Assumption \ref{assum:epsilon_conv}) and thus $\frac{\Delta_k^\phi}{\Delta_k^\phi+4\epsilon_f}\geq \frac{1}{2}$.

Therefore, we have,
\begin{align*}
	\left(\frac{1}{\Delta_{k+1}^\phi} - \frac{1}{\Delta_0^\phi} \right) - \left(\frac{1}{\Delta_k^\phi} - \frac{1}{\Delta_0^\phi} \right)  = 
	\frac{1}{\Delta_{k+1}^\phi} - \frac{1}{\Delta_{k}^\phi} \geq \frac{c_1 \alpha_k(1-\theta)^2 }{4D^2} - \frac{4 \epsilon_f}{\epsilon^2},
\end{align*}
which completes the proof.}
\end{proof}

We now bound the amount of increase on \emph{false} and \emph{successful} iterations.

\begin{lemma} \label{lemm:convex_2} Let Assumptions \ref{assum:func_bnd}, \ref{assum:conv} and \ref{assum:epsilon_conv} hold, and consider any realization of Algorithm \ref{alg:grad_approx_ls}. For every iteration that is \textbf{false} and \textbf{successful}, we have 
\begin{align*}
	\ab{ z_{k+1}  \geq z_k} - \frac{4 \epsilon_f}{\epsilon^2}.
\end{align*}
\end{lemma}
\begin{proof}{  
For every \emph{false} and \emph{successful} iteration, by Lemma \ref{stepsize_threshold_lemma1_succ} we have
\begin{align*}
	\phi(x_{k+1}) &\leq \phi(x_k) - c_1 \alpha_k \| \ab{g_k} \|^2 + 4\epsilon_f\\
				& \leq \phi(x_k)  + 4\epsilon_f.
\end{align*}
The rest of the proof is essentially a simplified version of the proof of Lemma  \ref{lemm:convex}, where the right hand side 
in \eqref{eq:conv1}  is simply replaced with $0$. }
\end{proof}

Let,
\begin{align} \label{eq:h_r_convex}
	h(\alpha) = \frac{c_1 \alpha(1-\theta)^2 }{4D^2}, \qquad \text{and} \qquad r(\epsilon_f) = \frac{4 \epsilon_f}{\epsilon^2}.
\end{align}
By Lemmas \ref{stepsize_threshold_lemma1}, \ref{lemm:convex} and \ref{lemm:convex_2} and Assumption \ref{assum:epsilon_conv}, for any realization of Algorithm \ref{alg:grad_approx_ls} (which specifies the sequence $\{\alpha_k,z_k\}$) and $k < N_\epsilon$, we have:
\begin{enumerate}
	\item (Lemma \ref{lemm:convex}) If $k$ is a \emph{true} and \emph{successful} iteration, then 
	\begin{align*}
		z_{k+1} \geq z_k + h(\alpha_k) - r(\epsilon_f), \quad \text{and} \quad \alpha_{k+1} = \tau^{-1} \alpha_k.
	\end{align*}
	\item (Lemma \ref{stepsize_threshold_lemma1}) If $\alpha_k \leq \bar{\alpha}$ and iteration $k$ is \emph{true}, then it is also \emph{successful}.
	\item (Lemma \ref{lemm:convex_2}) If $k$ is a \emph{false} and \emph{successful} iteration, then
	\begin{align*}
		z_{k+1} \geq z_k  - r(\epsilon_f).
	\end{align*}
	\item (Assumption \ref{assum:epsilon_conv}) $ \frac{ r(\epsilon_f)}{ h(\bar{\alpha})}<\gamma $ for  $\gamma\in (0,1)$.
\end{enumerate}
Hence, Assumption \ref{ass:alg_behave} holds, with $\bar{\alpha}>0$ defined in \eqref{alpha_bd1}, and $h(\Alpha_k)$ and $r(\epsilon_f)$ defined in \eqref{eq:h_r_convex}.

We now use Theorem \ref{th:mainbound} and the definitions of $\bar{\alpha}$, $h(\bar{\alpha})$, $r(\epsilon_f)$ and $Z_\epsilon$ to bound $\mathbb{E}[N_\epsilon]$.
\begin{theorem} \label{th:convex}
Let Assumptions \ref{assum:lip_cont}, \ref{assum:func_bnd}, \ref{assum:accurate} and \ref{assum:conv} hold. Moreover, let Assumption \ref{assum:epsilon_conv} hold, i.e., 
\begin{align*}
	\epsilon^2 > \max \left\{\frac{8 \epsilon_f LD^2}{\gamma c_1 (1-\theta)(1-2\theta - c_1(1-\theta))} , 16 \epsilon_f^2\right\},
\end{align*}
with the same  $\gamma\in(0,1)$ as used in Assumption  \ref{assum:accurate}.
 Then, the expected number of iterations that Algorithm \ref{alg:grad_approx_ls} takes until $ \phi(X_k) - \phi^\star \leq \epsilon$ occurs is bounded as follows
\begin{align*}
	\mathbb{E}[N_\epsilon] \leq \frac{2(1-\delta)}{(1-2\delta)^2-\gamma} \left[ M\left( \frac{1}{\epsilon} - \frac{1}{\phi(x_0) - \phi^\star} \right) + (1-\gamma)\log_\tau \left( \frac{\bar{\alpha}}{\alpha_0} \right)\right],
\end{align*}
where $M = \frac{4LD^2}{ c_1 (1-\theta)\left(1-2\theta - c_1(1-\theta)\right) } $.
\end{theorem}

\begin{remark}  If $\delta=\theta=\epsilon_f = 0$ our algorithm reduces to a deterministic line search algorithm with exact function evaluations and gradients. When $\epsilon_f=0$, $\gamma$ can be chosen arbitrarily small, and the lower bound on $\epsilon$ is $0$. Notice that the complexity bound has two components, the first component $\frac{8D^2L}{c_1 (1 - c_1)\epsilon} $ achieves its minimum value, $\frac{32D^2L}{\epsilon} $, for $c_1=1/2$ and is similar to the complexity bounds of the fixed step gradient descent method for convex functions, and the second term $ \log_\tau \left( \frac{2\left(1 - c_1\right)}{\alpha_0L}\right)$ bounds the total number of unsuccessful iterations, on which $\alpha_k$ is reduced.
\end{remark}

\subsubsection{Strongly Convex Functions}

In this section, we analyze the expected complexity of Algorithm \ref{alg:grad_approx_ls} in the case when $\phi$ is a strongly convex function. 
\begin{assumption}\label{assum:strongly} \textbf{(Strong convexity of $\pmb{\phi}$)} There exists a positive constant $\mu$ such that 
\begin{align*}
	\phi(x)\geq \phi(y)+\nabla \phi(y)^T(x-y)+\frac{\mu}{2}\|x-y\|^2, \quad \text{for all } x,y \in \mathbb{R}^n.
\end{align*}
\end{assumption}
Under Assumption \ref{assum:strongly}, let $\phi^\star = \phi(x^\star)$, where $x^\star$ is the minimizer of $\phi$. 

Recall the definition of $\Delta_k^\phi$ \eqref{eq:Delta_k}. In this setting, we bound the number of iterations taken by Algorithm \ref{alg:grad_approx_ls}  until $\Delta_k^\phi \leq \epsilon$ occurs. However, in this setting $Z_k$ is defined as $Z_k=\log\left(\frac{1}{\Delta_k^\phi}\right)$ and the resulting complexity  bound is logarithmic in $\frac{1}{\epsilon}$. Note, similar to the convex case, due to the noise in the function evaluations, $\epsilon$ cannot be chosen to be arbitrarily small. We give a precise lower bound on $\epsilon$, and thus explicitly derive a bound for the neighborhood of convergence.
\begin{assumption} \label{assum:epsilon_stronglyconv} 
\textbf{(Neighborhood of convergence, strongly convex case)} 
\begin{align*}
	\epsilon > \frac{4 \epsilon_f}{\left( 1 - \frac{2 \mu c_1 (1-\theta)(1-2\theta - c_1(1-\theta))}{L} \right)^{-\gamma} - 1},
\end{align*}
with the same  $\gamma\in(0,1)$ as used in Assumption  \ref{assum:accurate}. 
\end{assumption}

\begin{remark} 
The above assumption  again implies Assumption \ref{ass:alg_behave}(iv) with the same 
constant $\gamma$, which   connects  $\epsilon_f$ to the lower bound on $\epsilon$. Again, $\gamma$ can be chosen to be $\frac{1}{2}$, for simplicity. 
\end{remark}

By Lemma \ref{stepsize_threshold_lemma1}, whenever $\Alpha_k \leq \bar{\alpha}$, then every \emph{true} iteration is also \emph{successful}. We now show that on \emph{true} and \emph{successful} iterations, $Z_k$ increases by at least some function $h(\Alpha_k) - r(\epsilon_f)$, for all $k < N_\epsilon$.

\begin{lemma} \label{lemm:stronglyconvex} 
Let Assumptions \ref{assum:func_bnd}, \ref{assum:strongly} and \ref{assum:epsilon_stronglyconv} hold, and consider any realization of Algorithm \ref{alg:grad_approx_ls}. For every iteration that is \textbf{true} and \textbf{successful}, we have
\begin{align*} 
	\ab{ z_{k+1}  \geq z_k} - \log \left( 1 - \mu c_1 \alpha_k  (1-\theta)^2 \right) - \log \left( 1 + \frac{4\epsilon_f}{\epsilon} \right). 
\end{align*}
\end{lemma}

\begin{proof}{ Assumption \ref{assum:strongly} implies that ($x = x_k$ and $y = x^\star$)
\begin{align*}
	\phi(x_k) - \phi^\star \leq \frac{1}{2\mu} \| \nabla \phi(x_k)\|^2,
\end{align*}
see \cite[Theorem 2.1.10]{Nesterov}. Equivalently, using \eqref{eq:boundgk}
\begin{align*}
	\| \ab{g_k} \|^2 \geq (1-\theta)^2 \| \nabla \phi(x_k)\|^2 \geq 2\mu (1-\theta)^2 (\phi(x_k) - \phi^\star).
\end{align*}
By equation \eqref{eq.decrease_phi}, for every \emph{true} and \emph{successful} iteration we have
\begin{align}\label{eq.decrease_phi_conv}
	\phi(x_{k+1}) & \leq \phi(x_k) - c_1 \alpha_k (1-\theta)^2 \| \nabla \phi(x_k)\|^2 + 4\epsilon_f \nonumber\\
					& \leq \phi(x_k) - 2\mu c_1 \alpha_k  (1-\theta)^2 (\phi(x_k) - \phi^\star) + 4\epsilon_f,
\end{align}
and, thus,
\begin{align*}
	\phi(x_{k+1})  - \phi^\star \leq \left( 1- 2\mu c_1 \alpha_k  (1-\theta)^2 \right) (\phi(x_k) - \phi^\star) + 4\epsilon_f.
\end{align*}
Since we have that $\phi(x_k) - \phi^\star \geq \epsilon$,
\begin{align*}
	\phi(x_{k+1})  - \phi^\star &\leq \left( 1- 2\mu c_1 \alpha_k  (1-\theta)^2 \right) (\phi(x_k) - \phi^\star) + 4\epsilon_f\\
				& \leq \left( 1- 2\mu c_1 \alpha_k  (1-\theta)^2 \right) (\phi(x_k) - \phi^\star) + \frac{4\epsilon_f}{\epsilon}(\phi(x_k) - \phi^\star)\\
				& = \left( 1- 2\mu c_1 \alpha_k  (1-\theta)^2 + \frac{4\epsilon_f}{\epsilon}\right) (\phi(x_k) - \phi^\star).
\end{align*}
Thus, using the definition of $\Delta_k^\phi$, we have
\begin{align*}
	\Delta_{k+1}^\phi \leq  \left( 1- 2\mu c_1 \alpha_k  (1-\theta)^2 + \frac{4\epsilon_f}{\epsilon}\right) \Delta_k^\phi.
\end{align*}
Since $\epsilon > 4\epsilon_f$ (due to Assumption \ref{assum:epsilon_stronglyconv}), we have 
\begin{align*}
	\Delta_{k+1}^\phi &\leq  \left( 1- 2\mu c_1 \alpha_k  (1-\theta)^2 + \frac{4\epsilon_f}{\epsilon}\right) \Delta_k^\phi \nonumber\\
	&\le \left( 1- \mu c_1 \alpha_k  (1-\theta)^2  - \frac{4\epsilon_f}{\epsilon} \mu c_1 \alpha_k  (1-\theta)^2 + \frac{4\epsilon_f}{\epsilon}\right) \Delta_k^\phi \nonumber\\
	&= \left( 1- \mu c_1 \alpha_k  (1-\theta)^2 \right) \left( 1 + \frac{4\epsilon_f}{\epsilon} \right) \Delta_k^\phi. 
\end{align*}
Notice that since $\left( 1 + \frac{4\epsilon_f}{\epsilon} \right) > 0$, $\Delta_k^\phi > 0$ and $\Delta_{k+1}^\phi \geq 0$, this implies that $ 1- \mu c_1 \alpha_k  (1-\theta)^2 \geq 0$. Now taking the inverse and then the log of both sides and addinh $\log \Delta_0^\phi$, we have
\begin{align*}
	\log \left( \frac{\Delta_0^\phi}{\Delta_{k+1}^\phi} \right) \ge \log \left( \frac{\Delta_0^\phi}{\Delta_k^\phi} \right) - \log \left( 1 - \mu c_1 \alpha_k  (1-\theta)^2 \right) - \log \left( 1 + \frac{4\epsilon_f}{\epsilon} \right),
\end{align*}
which completes the proof.
}
\end{proof}

We note here that $ 1- \mu c_1 \alpha_k  (1-\theta)^2 \geq 0$ holds for all  $\alpha_k \leq \bar\alpha$ due to the constraint $\theta \in \left[0, \frac{1-c_1}{2-c_1}\right)$. 

We now bound the amount of increase on \emph{false} and \emph{successful} iterations.

\begin{lemma} \label{lemm:stronglyconvex_2} 
Let Assumptions \ref{assum:func_bnd}, \ref{assum:strongly} and \ref{assum:epsilon_stronglyconv} hold, and consider any realization of Algorithm \ref{alg:grad_approx_ls}. For every iteration that is \textbf{false} and \textbf{successful}, we have
\begin{align*}
	\ab{ z_{k+1}  \geq z_k} - \log \left( 1 + \frac{4 \epsilon_f}{\epsilon}\right).
\end{align*}
\end{lemma}

\begin{proof}{ For every \emph{false} and \emph{successful} iteration, by Lemma \ref{stepsize_threshold_lemma1_succ} we have
\begin{align*}
	\phi(x_{k+1}) &\leq \phi(x_k) - c_1 \alpha_k \| \ab{g_k} \|^2 + 4\epsilon_f.
\end{align*}

The rest of the proof is essentially a simplification of the proof of Lemma \ref{lemm:stronglyconvex} with the middle term of the right hand side of \eqref{eq.decrease_phi_conv} replaced by $0$. 
}
\end{proof}
Let 
\begin{align}\label{eq:h_r_str_convex}
	h(\alpha) = - \log(1- \mu c_1   (1-\theta)^2 \alpha), \qquad \text{and } r(\epsilon_f) = \log \left( 1 + \frac{ 4\epsilon_f}{\epsilon}\right)
\end{align}

By Lemmas \ref{stepsize_threshold_lemma1}, \ref{lemm:stronglyconvex} and \ref{lemm:stronglyconvex_2} and Assumption \ref{assum:epsilon_stronglyconv}, for any realization of Algorithm \ref{alg:grad_approx_ls} (which specifies the sequence $\{\alpha_k,z_k\}$) and $k < N_\epsilon$, we have:
\begin{enumerate}
	\item (Lemma~\ref{lemm:stronglyconvex}) If $k$ is a \emph{true} and \emph{successful} iteration, then
	\begin{align*}
		z_{k+1} \geq z_k + h(\alpha_k) - r(\epsilon_f), \quad \text{and} \quad \alpha_{k+1} = \tau^{-1} \alpha_k.
	\end{align*}
	\item (Lemma~\ref{stepsize_threshold_lemma1}) If $\alpha_k \leq \bar{\alpha}$ and iteration $k$ is \emph{true}, then it is also \emph{successful}.
	\item (Lemma~\ref{lemm:stronglyconvex_2}) If $k$ is a \emph{false} and \emph{successful} iteration, then
	\begin{align*}
		z_{k+1} \geq z_k  - \log \left( 1 + \frac{ 4\epsilon_f}{\epsilon}\right).
	\end{align*}
	\item(Assumption~\ref{assum:epsilon_stronglyconv}) $\frac{ r(\epsilon_f)}{h(\bar{\alpha})}<\gamma$ for some $\gamma \in (0,1)$. 
\end{enumerate}
Hence, Assumption \ref{ass:alg_behave} holds, with $\bar{\alpha}>0$ defined in \eqref{alpha_bd1}, and $h(\Alpha_k)$ and $r(\epsilon_f)$ defined in \eqref{eq:h_r_str_convex}. 

We now use Theorem \ref{th:mainbound} and the definitions of $\bar{\alpha}$, $h(\bar{\alpha})$, $r(\epsilon_f)$ and $Z_\epsilon$ to bound $\mathbb{E}[N_\epsilon]$.
\begin{theorem} \label{th:str_convex}
Let Assumptions \ref{assum:lip_cont}, \ref{assum:func_bnd}, \ref{assum:accurate} and \ref{assum:strongly} hold. Moreover, let Assumption \ref{assum:epsilon_stronglyconv} hold, i.e., 
\begin{align*}
	\epsilon > \frac{4 \epsilon_f}{\left( 1 - \frac{2 \mu c_1 (1-\theta)(1-2\theta - c_1(1-\theta))}{L} \right)^{-\gamma} - 1},
\end{align*}
with the same  $\gamma\in(0,1)$ as used in Assumption  \ref{assum:accurate}.
Then the expected number of iterations that Algorithm \ref{alg:grad_approx_ls} takes until $ \phi(X_k) - \phi^\star \leq \epsilon$ occurs is bounded as follows
\begin{align*}
	\mathbb{E}[N_\epsilon] \leq  \frac{2(1-\delta)}{(1-2\delta)^2 - \gamma} \left[ 2 \log_{1/M}\left( \frac{\phi(x_0) - \phi^\star}{\epsilon}\right) + (1-\gamma)\log_\tau \left( \frac{\bar{\alpha}}{\alpha_0} \right)\right],
\end{align*}
where $M =  1 -   \frac{2\mu c_1(1-\theta) \left(1-2\theta - c_1(1-\theta)\right)}{L} $.
\end{theorem}

\begin{remark} Again, if $\delta=\theta=\epsilon_f = 0$ our algorithm reduces to a deterministic line search algorithm with exact function evaluations and gradients. The complexity bound has two components, $4\log_{1/M} \left(\frac{1}{\epsilon}\right) $ where $M=1-  \frac{4\mu c_1\left(1- c_1\right)}{L}$  achieves its minimum value, $1-\frac{\mu}{L} $, for $c_1=1/2$ and  is similar to complexity bounds of the fixed step gradient descent method for strongly convex functions, and the second term again is the bound on the total number of unsuccessful iterations.
\end{remark}

\subsubsection{Nonconvex Functions}

In this section, we analyze the expected complexity of Algorithm \ref{alg:grad_approx_ls} in the case when $\phi$ is a nonconvex function. Again, we first specify the neighborhood of convergence. In this setting $Z_k = \phi(X_0) - \phi(X_k)$.
\begin{assumption} \label{assum:epsilon_nonconv} \textbf{(Neighborhood of convergence, nonconvex case)}  
\begin{align*}
		\epsilon^2 > \frac{2\epsilon_f L}{\gamma c_1 (1-\theta)(1-2\theta-c_1(1-\theta))},
\end{align*}
with the same  $\gamma\in(0,1)$ as used in Assumption  \ref{assum:accurate}.
\end{assumption}

\begin{remark}
The role of $\gamma$ is the same as in the convex and strongly convex cases.
\end{remark}

Let 
\begin{align} \label{eq:h_r_nonconvex}
	h(\alpha) = c_1 \alpha (1-\theta)^2 \| \nabla \phi (x_k)\|^2, \qquad \text{and} \qquad r(\epsilon_f) = 4\epsilon_f.
\end{align}

By Lemmas \ref{stepsize_threshold_lemma1} and \ref{stepsize_threshold_lemma1_succ} and Assumption \ref{assum:epsilon_nonconv}, for any realization of Algorithm \ref{alg:grad_approx_ls} (which specifies the sequence $\{ \alpha_k, z_k\}$) and $k < N_\epsilon$, we have:
\begin{enumerate}
	\item (Lemma~\ref{stepsize_threshold_lemma1}) If $k$ is a \emph{true} and \emph{successful} iteration, then
	\begin{align*}
		z_{k+1} \ge z_k + h(\alpha_k) - r(\epsilon_f) \quad \text{and} \quad \alpha_{k+1} = \tau^{-1} \alpha_k
	\end{align*}
	\item (Lemma~\ref{stepsize_threshold_lemma1}) If $\alpha_k \leq \bar{\alpha}$ and iteration $k$ is \emph{true}, then it is also \emph{successful}.
	\item (Lemma~\ref{stepsize_threshold_lemma1_succ}) If $k$ is a \emph{false} and \emph{successful} iteration, then
	\begin{align*}
		z_{k+1} \geq z_k  - 4 \epsilon_f.
	\end{align*}
	\item (Assumption~\ref{assum:epsilon_nonconv}) $\frac{ r(\epsilon_f)}{h(\bar{\alpha})}<\gamma$
	 for some $\gamma \in (0,1)$.
\end{enumerate}
Hence, Assumption \ref{ass:alg_behave} holds, with $\bar{\alpha}>0$ defined in \eqref{alpha_bd1}, and $h(\Alpha_k)$ and $r(\epsilon_f)$ defined in \eqref{eq:h_r_nonconvex}.

We now use Theorem \ref{th:mainbound} and the definitions of $\bar{\alpha}$, $h(\bar{\alpha})$, $r(\epsilon_f)$ and $Z_\epsilon$ to bound $\mathbb{E}[N_\epsilon]$.
\begin{theorem} \label{th:nonconvex}
Let Assumptions \ref{assum:lip_cont}, \ref{assum:low_bound}, \ref{assum:func_bnd} and \ref{assum:accurate}. Moreover, let Assumption \ref{assum:epsilon_nonconv} hold, i.e., 
\begin{align*}
	\epsilon^2 > \frac{ 2 \epsilon_f L}{\gamma c_1(1-\theta)(1-2\theta - c_1(1-\theta))},
\end{align*}
with the same  $\gamma\in(0,1)$ as used in Assumption  \ref{assum:accurate}.
Then, the expected number of iterations that Algorithm \ref{alg:grad_approx_ls} takes until $\| \nabla \phi(X_k) \| \leq \epsilon$ occurs is bounded as follows
\begin{align*}
	\mathbb{E}[N_\epsilon] \leq \frac{2(1-\delta)}{(1 - 2 \delta)^2-\gamma} \left[ \frac{M}{\epsilon^2} + (1-\gamma)\log_\tau \left( \frac{\bar{\alpha}}{\alpha_0}\right)\right],
\end{align*}
where $M = \frac{(\phi(x_0) - \hat{\phi})L}{c_1 (1-\theta) \left(1-2\theta - c_1(1-\theta)\right) } $.
\end{theorem}

\begin{remark} Again, if $\delta=\theta=\epsilon_f = 0$ our algorithm reduces to a deterministic line search with the exact gradients. The complexity bound has two components, $\frac{2M}{ \epsilon^2} $ where $M=\frac{(\phi(x_0) - \hat{\phi})L}{c_1 \left(1- c_1\right) } $  achieves its minimum value, $4(f(x_0) - \hat{f})L $, for $c_1=1/2$ and  is similar to complexity bounds of the fixed step gradient descent for nonconvex functions, and the second term, as before, is the bound on the total number of unsuccessful iterations.
\end{remark}

\subsection{General Descent}\label{sec.gen_descent}

For simplicity, in the analysis of the previous sections we assumed that the search direction at every iteration was defined as $d_k = -\ab{g_k}$. Here, we show how our analysis can be extended to account for more general search direction, e.g., quasi-Newton search direction where $d_k = -H_k\ab{g_k}$ \cite{NoceWrig06}, provided the search directions satisfy:
\begin{itemize}
	\item There exists a constant $\beta>0$, such that:
	\begin{align} \label{eq:beta}
		\frac{d_k^T \ab{g_k}}{\| d_k\| \| \ab{g_k}\|} \leq - \beta, \qquad \text{for all $k$,}
	\end{align}
	\item There exist constants $\kappa_1, \kappa_2 >0$, such that:
	\begin{align}	\label{eq:kappa}
		\kappa_1 \| \ab{g_k}\| \leq \| d_k \| \leq \kappa_2 \| \ab{g_k}\|, \qquad \text{for all $k$,}
	\end{align}
\end{itemize}
together with \eqref{eq:theta_cond}. Of course, in this setting, the modified line search would be given by \eqref{eq:Armijo}, and the convergence analysis would have dependence on $\beta$, $\kappa_1$ and $\kappa_2$. 

All we need to do is to derive an expression for $\bar \alpha$ for the general search direction case, and prove analogues of Lemmas \ref{stepsize_threshold_lemma1} and \ref{stepsize_threshold_lemma1_succ}. First, we change the bound on $\theta$ in the Definition \ref{def:suff_acc}. In particular we will require that $\theta\in\left [0, \frac{(1-c_1)\beta}{1+(1-c_1)\beta}\right )$. Now we can prove the following lemma. 

\begin{lemma}	\label{alphabar_d}
Let Assumption \ref{assum:lip_cont}  hold. For every realization  of Algorithm \ref{alg:grad_approx_ls},  if iteration $k$ is \textbf{true} (i.e., $I_k=1$), and if 
\begin{equation}\label{alpha_bd2}
	\alpha_k \leq \bar{\alpha} = \frac{2}{L \kappa_2} \left[ \frac{(1-c_1)(1-\theta)\beta - \theta}{1-\theta} \right],
\end{equation}
then \eqref{eq:Armijo} holds. In other words, when \eqref{alpha_bd2} holds, any \textbf{true} iteration is also a \textbf{successful} iteration. Moreover, for every \textbf{true} and \textbf{successful} iteration,
\begin{align*}		
	\phi(x_{k+1}) &\leq \phi(x_k) - c_1 \alpha_k \beta \kappa_1 (1-\theta)^2\| \nabla \phi(x_k) \|^2 + 4\epsilon_f.
\end{align*}
\end{lemma}

\begin{proof}
The proof is very similar to that of Lemma \ref{stepsize_threshold_lemma1}. First from Assumption \ref{assum:lip_cont}, we have
\begin{align*}
\phi(x_{k+1}) \le \phi(x_k) + \alpha_k d_k^T \nabla \phi(x_k) + \frac{L}{2} \|\alpha_k d_k\|^2.
\end{align*}
Applying the Cauchy-Schwarz inequality, \eqref{eq:theta_cond} and \eqref{eq:boundgk}, for every \emph{true} iteration 
\begin{align*}
	\phi(x_k + \alpha_k d_k) & \leq \phi(x_k) + \alpha_k d_k^T \nabla \phi (x_k) + \frac{\alpha_k^2 L}{2} \| d_k \|^2\\
			& = \phi(x_k) + \alpha_k d_k^T (\nabla \phi (x_k) - \ab{g_k}) + \alpha_k d_k^T \ab{g_k} + \frac{\alpha_k^2 L}{2} \| d_k \|^2\\
			& \leq \phi(x_k) + \alpha_k \| d_k \| \|\nabla \phi (x_k) - \ab{g_k}\| + \alpha_k d_k^T \ab{g_k}) + \frac{\alpha_k^2 L}{2} \| d_k \|^2\\
			& \leq \phi(x_k) + \frac{\alpha_k \theta}{1 - \theta} \|d_k\| \|\ab{g_k}\| + \alpha_k d_k^T \ab{g_k} + \frac{\alpha_k^2 L \kappa_2}{2} \| d_k \| \| \ab{g_k} \| \\
			& \leq \phi(x_k) + \alpha_k d_k^T \ab{g_k} + \alpha_k \left[ \frac{\theta}{1 - \theta} + \frac{\alpha_k L \kappa_2}{2} \right] \| d_k \| \| \ab{g_k} \|.
\end{align*}
Now, using Assumption \ref{assum:func_bnd}, we have
\begin{align*}
	f(x_k + \alpha_k d_k,\ab{\xi}) \leq f(x_k,\ab{\xi}) + \alpha_k d_k^T \ab{g_k} + \alpha_k \left[ \frac{\theta}{1 - \theta} + \frac{\alpha_k L \kappa_2}{2} \right] \| d_k \| \| \ab{g_k} \| + 2 \epsilon_f.
\end{align*}
From this we conclude that \eqref{eq:Armijo} holds whenever 
\begin{align*}
	&f(x_k,\ab{\xi}) + \alpha_k d_k^T \ab{g_k} + \alpha_k \left[ \frac{\theta}{1 - \theta} + \frac{\alpha_k L \kappa_2}{2} \right] \| d_k \| \| \ab{g_k} \| + 2 \epsilon_f \\
	& \qquad \leq f(x_k,\ab{\xi}) + c_1 \alpha_k d_k^T \ab{g_k} + 2\epsilon_f,
\end{align*}
or equivalently, since $\alpha_k>0$,
\begin{align*}
	\left[ \frac{\theta}{1 - \theta} + \frac{\alpha_k L \kappa_2}{2} \right] \| d_k \| \| \ab{g_k} \| &\le - (1-c_1) d_k^T \ab{g_k}.
\end{align*}
Using \eqref{eq:beta}, the above expression holds whenever $\alpha_k$ satisfies \eqref{alpha_bd2}. Therefore, using Assumption \ref{assum:func_bnd}, \eqref{eq:kappa}, and \eqref{eq:boundgk}, for every \emph{true} and \emph{successful} iteration we have
\begin{align*}
	\phi(x_{k+1}) &\leq \phi(x_k) - c_1 \alpha_k \beta \kappa_1 (1-\theta)^2\| \nabla \phi(x_k) \|^2 + 4\epsilon_f, 
\end{align*}
which completes the proof.
\end{proof}

Next, we state and prove a result for the case of \emph{false} and \emph{successful} iterations.
\begin{lemma}
For every \textbf{false} and \textbf{successful} iteration of Algorithm \ref{alg:grad_approx_ls},  we have
\begin{align*}
	\phi(x_{k+1}) &\leq \phi(x_k) - c_1 \beta \alpha_k \kappa_1 \|\ab{g_k}\|^2 + 4\epsilon_f.
\end{align*}
\end{lemma}

\begin{proof}
For every \emph{successful} iteration we have
\begin{align*}
	f(x_{k+1},\ab{\xi}) &\leq f(x_k,\ab{\xi}) + c_1 \alpha_k d_k^T \ab{g_k} + 2\epsilon_f.
\end{align*}
Thus, by Assumption \ref{assum:func_bnd}, \eqref{eq:kappa}, and \eqref{eq:boundgk}
\begin{align*}	
	\phi(x_{k+1}) &\leq \phi(x_k) + c_1 \alpha_k d_k^T \ab{g_k} + 4\epsilon_f \\
	&\le \phi(x_k) - c_1 \alpha_k \beta \|d_k\| \|\ab{g_k}\| + 4\epsilon_f \\
	&\le \phi(x_k) - c_1 \alpha_k \beta \kappa_1 \|\ab{g_k}\|^2 + 4\epsilon_f, 
\end{align*}
which is a repetition of the last part of the proof of Lemma~\ref{alphabar_d}. 
\end{proof}

The rest of the analysis (deriving expected complexity bounds) applies almost without change, taking into account the influence of the constants $\beta$,  $\kappa_1$ and $\kappa_2$. 

\ab{
\subsection{Convergence under Condition \eqref{eq:alpha_cond_mod}}\label{sec:conv_alpha} 
{
In this section we demonstrate how our analysis can be extended to a different setting in terms of gradient estimate computations. To avoid introducing new notation, we will keep the discussion at a high level, which will hopefully be clear to the reader. The precise derivations in this sections are straightforward  extensions of derivations above. 

As we have pointed out before, the key condition \eqref{eq:theta_cond} can be satisfied by various gradient approximation schemes discussed in  \cite{berahas2019theoretical}. However, all these schemes require ${\cal O}(n)$ function evaluations to obtain $g_k$ that satisfies \eqref{eq:theta_cond}. This can be expensive in a high dimensional setting. On the other hand, in many applications a stochastic estimate of $\nabla \phi(x)$ may be directly available and thus $g_k$ can be computed by a sample averaging scheme. Since we assume that the function values are computed with noise, we cannot assume that these stochastic estimates are unbiased\change{. However,} as in the case of the function noise, we can assume that this bias is bounded.

\begin{assumption}	\label{assum:epsilon_g} 
\textbf{(Biased gradient estimates)} 
For each  $x$, we have an ability to compute a random vector $h(x,\xi)$, which is a (possibly) biased estimate of $\nabla \phi(x_k)$, and the bias is bounded by a known constant $\epsilon_g$, i.e., for all $x$
\begin{align*}
	\| \mathbb{E}[h(x,\xi)] - \nabla \phi(x) \| \leq \epsilon_g
\end{align*}
where the expectation is over random variable $\xi$. 
\end{assumption}

Thus, for any $\zeta>1$, by averaging a sufficiently large number of samples $h(x,\xi)$ we can compute a (random) $g$ such that $\| \ab{g} - \nabla \phi(x) \| \leq  {\zeta \epsilon_g}$,  with sufficiently high probability. 
On the other hand, without knowing  $\|\nabla \phi(x_k)\|$ we cannot ensure  \eqref{eq:theta_cond}. 
Here, we present the outline of the analysis of our modified line search method where \eqref{eq:theta_cond} is replaced with a condition 
\begin{align*}
	\| \ab{g_k} - \nabla \phi(x_k) \| \leq \max \{ {\zeta \epsilon_g},\kappa\alpha_k \|  \ab{g_k} \| \}.
\end{align*}
for some $\zeta > 1$,  and $\kappa \geq 0$.  Essentially, we want to relax   \eqref{eq:theta_cond} as long as $\kappa\alpha_k \|  \ab{g_k} \|$ is not so small that 
$\| \ab{g_k} - \nabla \phi(x_k) \| \leq \kappa\alpha_k \|  \ab{g_k} \|$ cannot be enforced with sufficiently high probability. When this happens, we want  \eqref{eq:theta_cond} to hold, 
which we can ensure by 
$\| \ab{g_k} - \nabla \phi(x_k) \| \leq  {\zeta \epsilon_g}$, as long as {$\|\nabla \phi(x_k) \| > \frac{\zeta\epsilon_g}{\theta}$}. Thus we need to add this lower bound on the gradient to our definition of the stopping time: 

\begin{definition}
\
\begin{itemize}\label{def:n_epsilon}
	\item If \ab{$\phi$} is convex or strongly convex: $N_\epsilon$ is the number of iterations required until either $\phi(X_k) - \phi^\star\ \leq \epsilon$ or {$\|\nabla \phi(x_k) \| \leq  \frac{\zeta\epsilon_g}{\theta}$}
	occurs for the first time. Note, $\phi^\star = \phi(x^\star)$, where $x^\star$ is a global minimizer of \ab{$\phi$}.
	\item If \ab{$\phi$} is nonconvex: $N_\epsilon$ is  the number of iterations required until $\| \nabla \phi(X_k)\| \leq \max\{\epsilon, \frac{\zeta\epsilon_g}{\theta}\}$ occurs for the first time. 
\end{itemize}
\end{definition}

}

For brevity, in this section we do not derive all the results, or state all the intermediate lemmas, rather we state the key results, without proof. We first present the analogue of Definition \ref{def:suff_acc} where \eqref{eq:theta_cond} is replaced with \eqref{eq:alpha_cond_mod}.
\begin{definition} \label{def:suff_acc_2}
A sequence of random gradients $\{G_k\}$ is $(1-\delta)$-probabilistically ``sufficiently accurate'' for Algorithm \ref{alg:grad_approx_ls} 
 if there exists constants {$\zeta > 1$} and $\kappa \geq 0$, such that the indicator variables
\begin{align*}
	I_k = \mathbbm{1}\{ \| G_k - \nabla \phi(X_k) \| \leq \max \{ {\zeta \epsilon_g},\kappa\Alpha_k \|G_k\|\}\}
\end{align*}
satisfy the following submartingale condition
\begin{align*}
	\mathbb{P}(I_k = 1 | \mathcal{F}_{k-1}^{G, \mathcal{E}} ) \geq 1 - \delta,
\end{align*}
\ab{for all realizations of $\mathcal{F}_{k-1}^{G, \mathcal{E}}$, }where $\mathcal{F}_{k-1}^{G, \mathcal{E}}  = \sigma(G_0,\ldots,G_{k-1},\mathcal{E}_{k-1})$ is the $\sigma$-algebra generated by $G_0,\ldots,G_{k-1}$ and  
$\mathcal{E}_{k-1}$, \ab{for all realizations}. Moreover, we say that iteration $k$ is a \textbf{true} iteration if the event $I_k = 1$ occurs, otherwise the iteration is called \textbf{false}.
\end{definition}

We assume (as was done in Section \ref{sec:conv_norm}) that Assumption \ref{assum:accurate} holds for Definition \ref{def:suff_acc_2}. In order to prove expected complexity bounds under  \eqref{eq:alpha_cond}, we make the following minor modification to Algorithm \ref{alg:ls_sub}. When the step is successful, $\alpha_{k+1} = \min \{ \tau^{-1}\alpha_k,\alpha_{\max}\}$, where $\alpha_{\max}>0$. 
\begin{lemma}\label{stepsize_threshold_lemma1_cond2}
Let Assumptions \ref{assum:lip_cont} and \ref{assum:func_bnd}  hold. For every realization  of Algorithm \ref{alg:grad_approx_ls},  if iteration $k$ is \textbf{true} (i.e., $I_k=1$), and if 
\begin{equation}\label{alpha_bd1_cond2}
	\alpha_k \leq \bar{\alpha} = \min \left\{\frac{2(1-2\theta - c_1(1-\theta))}{L(1-\theta)},\frac{2(1-c_1)}{L + 2\kappa} \right\},
\end{equation}
then \eqref{eq:Armijo} holds. In other words, when \eqref{alpha_bd1_cond2} holds, any \textbf{true} iteration is also a \textbf{successful} iteration. Moreover, for every \textbf{true} and \textbf{successful} iteration,
\begin{align}		\label{eq.decrease_phi_cond2}
	\phi(x_{k+1}) &\leq \phi(x_k) - c_1 \alpha_k \min \left\{ (1 - \theta)^2, \frac{1}{(1+ \kappa \alpha_{\max})^2}\right\} \| \nabla \phi(x_k) \|^2 + 4\epsilon_f.
\end{align}
Furthermore, for every \textbf{false} and \textbf{successful} iteration of Algorithm \ref{alg:grad_approx_ls},  we have
\begin{align*}	
	\phi(x_{k+1}) &\leq \phi(x_k) - c_1 \alpha_k \| \ab{g_k} \|^2 + 4\epsilon_f.
\end{align*}
\end{lemma}
\change{We should note that, if $g_k$ is the true gradient we recover the step size parameter condition from the deterministic setting.}

We now present the complexity bounds for condition \ref{eq:alpha_cond_mod} for convex (Theorem \ref{th:convex_cond2}), strongly convex (Theorem \ref{th:str_convex_cond2}) and nonconvex functions (Theorem \ref{th:nonconvex_cond2}).
\begin{theorem} \label{th:convex_cond2}
Let Assumptions \ref{assum:lip_cont}, \ref{assum:func_bnd}, \ref{assum:epsilon_g}, \ref{assum:accurate} and \ref{assum:conv} hold. Moreover, let Assumption \ref{assum:epsilon_conv} hold, i.e., 
\begin{align*}
	\epsilon^2 > \max \left\{\frac{8 \epsilon_f D^2}{\gamma c_1 \min \left\{\frac{(1-\theta)(1-2\theta - c_1(1-\theta))}{L},\frac{1-c_1}{(L + 2\kappa)(1+\kappa \alpha_{\max})^2} \right\}} , 16 \epsilon_f^2 \right\},
\end{align*}
with the same  $\gamma\in(0,1)$ as used in Assumption  \ref{assum:accurate}.
 Then, the expected number of iterations that Algorithm \ref{alg:grad_approx_ls} takes until $ \phi(X_k) - \phi^\star \leq \epsilon$ {or $\| \nabla \phi(X_k) \| \leq \frac{\zeta\epsilon_g}{\theta}$} occurs is bounded as follows
\begin{align*}
	\mathbb{E}[N_\epsilon] \leq \frac{2(1-\delta)}{(1-2\delta)^2-\gamma} \left[ M\left( \frac{1}{\epsilon} - \frac{1}{\phi(x_0) - \phi^\star} \right) + (1-\gamma)\log_\tau \left( \frac{\bar{\alpha}}{\alpha_0} \right)\right],
\end{align*}
where $M = \frac{4D^2}{ c_1 \min \left\{\frac{(1-\theta)(1-2\theta - c_1(1-\theta))}{L},\frac{1-c_1}{(L + 2\kappa)(1+\kappa \alpha_{\max})^2} \right\} } $.
\end{theorem}
\begin{theorem} \label{th:str_convex_cond2}
Let Assumptions \ref{assum:lip_cont}, \ref{assum:func_bnd}, \ref{assum:epsilon_g}, \ref{assum:accurate} and \ref{assum:strongly} hold. Moreover, let Assumption \ref{assum:epsilon_stronglyconv} hold, i.e., 
\begin{align*}
	\epsilon >    \max \left\{   \frac{4 \epsilon_f}{\left( 1 - 2\mu c_1\min \left\{\frac{(1-\theta)(1-2\theta - c_1(1-\theta))}{L},\frac{1-c_1}{(L + 2\kappa)(1+\kappa \alpha_{\max})^2} \right\}  \right)^{-\gamma} - 1}   ,  { \frac{\zeta^2\epsilon_g^2}{2 \mu \theta^2}} \right\}     ,
\end{align*}
with the same  $\gamma\in(0,1)$ as used in Assumption  \ref{assum:accurate}.
Then the expected number of iterations that Algorithm \ref{alg:grad_approx_ls} takes until $ \phi(X_k) - \phi^\star \leq \epsilon$ or \change{$\| \nabla \phi(X_k) \| \leq \frac{\zeta\epsilon_g}{\theta}$} occurs is bounded as follows
\begin{align*}
	\mathbb{E}[N_\epsilon] \leq  \frac{2(1-\delta)}{(1-2\delta)^2 - \gamma} \left[ 2 \log_{1/M}\left( \frac{\phi(x_0) - \phi^\star}{\epsilon}\right) + (1-\gamma)\log_\tau \left( \frac{\bar{\alpha}}{\alpha_0} \right)\right],
\end{align*}
where $M =  1 -   2\mu c_1\min \left\{\frac{(1-\theta)(1-2\theta - c_1(1-\theta))}{L},\frac{1-c_1}{(L + 2\kappa)(1+\kappa \alpha_{\max})^2} \right\}$.
\end{theorem}
\change{
\begin{remark} 
In the last two theorems the bound on $\mathbb{E}[N_\epsilon] $ depends on $\epsilon$ but not on $\epsilon_g$. This bound should be understood as the bound on expected complexity to reach $\epsilon$-accuracy in terms of the function value. If $\| \nabla \phi(X_k) \| \leq \frac{\zeta\epsilon_g}{\theta}$ occurs before the  $\epsilon$-accuracy in the function value is reached, the bound clearly still holds. The next theorem derives the bound on the complexity of reaching  $\epsilon$-accuracy in terms of $\| \nabla \phi(X_k) \|$   which applies to convex and nonconvex functions, and has no direct implications on accuracy in terms of the function value. 
\end{remark}}
\begin{theorem} \label{th:nonconvex_cond2}
Let Assumptions \ref{assum:lip_cont}, \ref{assum:low_bound}, \ref{assum:func_bnd}, \ref{assum:epsilon_g} and \ref{assum:accurate}. Moreover, let Assumption \ref{assum:epsilon_nonconv} hold, i.e., 
\begin{align*}
	\epsilon^2 >  \max \left\{\frac{2 \epsilon_f}{\gamma c_1 \min \left\{\frac{(1-\theta)(1-2\theta - c_1(1-\theta))}{L},\frac{1-c_1}{(L + 2\kappa)(1+\kappa \alpha_{\max})^2} \right\}} , {\frac{\zeta^2 \epsilon_g^2 }{\theta^2}  }\right\} ,
\end{align*}
with the same  $\gamma\in(0,1)$ as used in Assumption  \ref{assum:accurate}.
Then, the expected number of iterations that Algorithm \ref{alg:grad_approx_ls} takes until $\| \nabla \phi(X_k) \| \leq \epsilon$ occurs is bounded as follows
\begin{align*}
	\mathbb{E}[N_\epsilon] \leq \frac{2(1-\delta)}{(1 - 2 \delta)^2-\gamma} \left[ \frac{M}{\epsilon^2} + (1-\gamma)\log_\tau \left( \frac{\bar{\alpha}}{\alpha_0}\right)\right],
\end{align*}
where $M = \frac{\phi(x_0) - \hat{\phi}}{c_1  \min \left\{\frac{(1-\theta)(1-2\theta - c_1(1-\theta))}{L},\frac{1-c_1}{(L + 2\kappa)(1+\kappa \alpha_{\max})^2} \right\}}$.
\end{theorem}

\begin{remark} If $\delta=\theta=\kappa=\epsilon_f=\epsilon_g = 0$ our algorithm reduces to a deterministic line search algorithm with exact function evaluations and gradients. The dependence on the target accuracy $\epsilon$ is the same as that of a deterministic line search algorithm.  
\end{remark} 

\begin{remark}
Independent of the condition used on the gradient accuracy (condition \ref{eq:theta_cond} or \ref{eq:alpha_cond}), the dependence on $\epsilon$ (the target accuracy) and $\delta$ (the probability of a true iteration) is the same. Moreover, in the setting where $\theta = \kappa = \change{\epsilon_g = } 0$, the results are identical. Finally, determining which condition is stronger is not trivial as it depends on the following iteration specific quantities $\|\nabla \phi(x_k)\|$, $\|g_k\|$ and $\alpha_k$.
\end{remark} 


}

\section{Final Remarks} \label{final_remarks}

We presented the analysis of a modified line search method that can be applied to functions with bounded noise, and where the gradient approximations $\ab{g_k}$ are possibly random, e.g., Gaussian smoothed  gradients \cite{nesterov2017random,ES} or sphere smoothed gradients \cite{flaxman2005online,fazel2018global}. However, as a special case, we recover results for gradient approximations that are not random ($\delta = 0$), e.g., finite difference approximations \cite{berahas2019derivative,kelley2011implicit} or linear interpolation gradient approximations \cite{ConnScheVice08c}.

Furthermore, we discuss the effect of the parameter $\gamma$, that plays a crucial role in the analysis presented. This parameter depends on the error in the function evaluations, and effectively controls the size of the neighborhood of convergence, i.e., the lower bound on the $\epsilon$. When there is zero error in the function evaluations, i.e., $\epsilon(x) = 0$ for all $x \in \mathbb{R}^n$, $\gamma$ can be chosen arbitrarily close to zero, in which case we recover the exact convergence results from \cite{cartis2018global}.

Finally, while our analysis assumes that the step size parameter is chosen using an adaptive line search procedure (Algorithm \ref{alg:ls_sub}), and thus varies at every iteration, it also holds for a constant step size parameter choice. Namely, if $\alpha_0 \leq \bar{\alpha}$ and $\tau = 1$, then $\alpha_k \leq \bar{\alpha}$ for all $k$, and all \emph{true} iterations are also \emph{successful} iterations. Thus, as a special case of the analysis presented in Section \ref{convergence_analysis}, we recover results for a fixed step size parameter procedure. We should note that the second term in the complexity bounds is zero in the case where $\tau = 1$ and $\alpha_0 = \bar{\alpha}$.

\ab{ We establish a bound on the expected number of iteration $N_\epsilon$ that the algorithm takes until it reaches  the desired near-optimal neighborhood. This is in contrast with the analyses of many other stochastic algorithms (such as stochastic gradient), where a bound is established on the expected ``proximity'' to the optimum (e.g., the expected smallest size of the gradient) achieved sometime during a given number of iterations. However, in all these cases there are no guarantees that the algorithm will remain in the near-optimal neighborhood, once it reaches it. To analyze the behavior of a stochastic algorithm near optimality  is a nontrivial task and requires considering the nature of the function in and near such a neighborhood. For example, for nonconvex functions, where the algorithm may  converge to a near-saddle point, it will very likely leave the neighborhood and never  return to it. On the other hand, if the objective function is strongly convex in the near-optimal neighborhood, then the algorithm is very likely to either stay in this neighborhood or keep returning to it frequently. Formally analyzing this behavior is the subject of a separate study. }

\bibliographystyle{siamplain}
\bibliography{Katya}



\end{document}